\documentclass[reqno]{amsart}

\usepackage{amssymb}
\usepackage{graphicx}

\usepackage{lipsum}
\parskip        \smallskipamount

\usepackage[usenames,dvipsnames]{xcolor}
\usepackage[colorlinks=true]{hyperref}
\def\blue{\color{blue}}
\def\red{\color{red}}

\definecolor{mygreen}{rgb}{0.05, 0.576, 0.03}           

\definecolor{mygray}{gray}{0.9}                                 

\definecolor{myred}{rgb}{0.768, 0.09, 0.09}

\theoremstyle{theorem}
\newtheorem{theorem}{Theorem}[section]
\newtheorem*{theorem*}{Theorem} 
\newtheorem{lemma}[theorem]{Lemma}
\newtheorem{corollary}[theorem]{Corollary}
\newtheorem{proposition}[theorem]{Proposition}
\newtheorem{ling}[theorem]{Linguistic convention}
\theoremstyle{definition}
\newtheorem{definition}[theorem]{Definition}
\theoremstyle{remark}
\newtheorem{remark}[theorem]{\bf\em Remark}

\def\N{\mathbb{N}}

\def\R{\mathbb{R}}

\def\P{\mathbb{P}}
\def\E{\mathbb{E}}

\def\PP{\mathscr{P}}

\def\LL{\mathscr{L}}
\def\RR{\mathscr{R}}
\def\TT{\mathscr{T}}
\def\CC{\mathscr{C}}
\def\EE{\mathscr{E}}
\def\VV{\mathscr{V}}
\def\HH{\mathscr{H}}

\renewcommand{\phi}{\varphi}
\renewcommand{\epsilon}{\varepsilon}

\newcommand{\1}{{\text{\Large $\mathfrak 1$}}}	%this clashes with cmtip packg

\newcommand{\var}{\operatorname{var}}

\renewcommand{\limsup}{\varlimsup}

\usepackage[margin=1cm]{caption}        %captions of all sorts
\numberwithin{equation}{section}

\usepackage{mathrsfs}
\def\bomega{{\boldsymbol\omega}}
\def\II{\mathscr I}
\def\JJ{\mathscr J}
\newcommand{\I}{{\text{\Large $\mathfrak 1$}}}	%this clashes with cmtip packg

\def\ee{\text{\footnotesize $\mathcal E$}}                                        
\def\eee{\text{\tiny $\mathcal E$}}

\def\ran{{\sf R}}
\def\dom{{\sf D}}

\newcommand{\jso}[1]{\stackrel{#1}{\cong}}
\usepackage{enumerate}                                                            
\usepackage{bold-extra}

\def\Y{\mathbb{Y}}

\definecolor{tempcol}{rgb}{0.10, 0.58, 0.99}

\usepackage{cancel}
\usepackage[normalem]{ulem}	%strike through \sout

\def\Deg{\operatorname{Deg}}
\def\xor{\text{\footnotesize$\triangle$}}

\def\Rado{\text{\textsf{\textbf{R}}}}

\usepackage{datetime2}
\DTMshowsecondsfalse                                                              
\DTMshowzonefalse

\newcommand{\tkmid}[2]{\begin{center}\begin{minipage}{#1} \color{red} \fbox{\small \tt #2} \end{minipage}\end{center}}

\def\pmax{{\widehat p}}
\def\SSS{\mathsf S}
\def\SSSone{\mathsf S_{\text{\rm I}}}
\def\SSStwo{\mathsf S_{\text{\rm II}}}
\def\TTT{\mathsf T}

\def\ZZ{\mathscr Z}

\usepackage{cases}
\usepackage{float}

\begin{document}

%% \noindent
%% {\red \footnotesize
%% Updated on \DTMnow ~
%% }

\title[Random graph isomorphisms]{A combinatorial approach to
phase transitions in random graph isomorphism problems}
\author{Dimitris Diamantidis} %\thanks{}
\author{Takis Konstantopoulos} %\thanks{}
\author{Linglong Yuan} %\thanks{}
\address{Math.\ Dept., Univ.\ of Liverpool, UK}
\email{Dimitris.Diamantidis@liverpool.ac.uk}
\email{takiskonst@gmail.com}
\email{yuanll@liverpool.ac.uk}
%\dedicatory{}

%    \subjclass is required.
%05C80   	Random graphs (graph-theoretic aspects) 
%05C60   	Isomorphism problems in graph theory (reconstruction conjecture, etc.) and homomorphisms (subgraph embedding, etc.)
%60C05   	Combinatorial probability
%05A19   	Combinatorial identities, bijective combinatorics
%60F99   	Limit Theorems None of the above, but in this section

%\today%date{20 Sept.\ 2024}

\begin{abstract}
We consider two independent Erd\H{o}s-R\'enyi random graphs,
with possibly different parameters, and study
%provide a rigorous mostly combinatorial approach regarding 
two isomorphism problems, a graph embedding problem and a common subgraph problem. 
Under certain conditions on the graph parameters we show a sharp
asymptotic phase transition as the graph sizes tend to infinity.
This extends known results for the case of uniform 
Erd\H{o}s-R\'enyi random graphs. 
Our approach is primarily combinatorial, naturally leading to several 
related problems for further exploration.
\iffalse
$G(m,p)$ and $G(n,q)$,
and ask two questions regarding isomorphisms. First, under what
conditions is there a isomorphic copy of the smallest one
into the second one?
Second, when do the two graphs contain a common (up to isomorphism)
induced subgraph?
In both cases there is a phase transition occurring at
a certain function $m_*(n)$ for the size of the smaller
graph as a function of the size $n$ of the bigger one,
such that the probability of either of the above events
converges to $1$ just below the threshold, or to $0$ just 
above it. We need to assume certain restrictions on $p$ and $q$
to achieve this. For the first question, we assume that $q=1/2$ but
$p$ arbitrary. For the second question, we assume that $(p,q)$
are in a certain neighborhood of the diagonal of the unit square.
The paper answers some open questions in \cite{ChD23}
and concludes with some further open questions.
\fi
\end{abstract}

\maketitle 

\hypersetup{linkcolor=Brown}
\tableofcontents

\section{\bfseries\scshape  Introduction}
Let $G(V,p)$ denote (the law of) a random undirected graph on the vertex set
$V$ in which an edge is present with probability $p$, 
independently from edge to edge.
When $V$ has finite size $n$ the random graph
is the well-known Erd\H{o}s-R\'enyi $G(n,p)$ graph,
but the term makes sense even when $V$ is infinite.
In particular, we let $G(\infty, p)$ be the law of this graph on a countably
infinite set of vertices. 

In a seminal paper, Erd\H{o}s and R\'enyi \cite{ER63} showed that
$G(\infty,1/2)$ is unique up to automorphisms.
More specifically, there is a (deterministic) graph $\Rado$,
unique up to isomorphisms, on a countable set of vertices and a
random bijection from the set of vertices of $G(\infty,1/2)$ 
onto the set of vertices of $\Rado$ that preserves edges. 

The graph $\Rado$ has a number of remarkable properties, as first shown by Rado
\cite{R64}. We refer to the papers of Cameron \cite{CAM97,CAM01}
for a survey of properties of $\Rado$ which
is now known by either of the following names:
{\em Rado Graph}, 
{\em Universal Graph} or [{\em sic}] {\em The Random Graph}.

$\Rado$ can be constructed in many ways. Here is one. 
Consider first von Neumann's representation of natural numbers
as finite ordinals. To the dismay of a probabilist, we use the
letter $\bomega$ for the set of natural numbers $0,1,2,\ldots$
with their standard order type.
This means that $0:=\varnothing$ and, recursively, $n := \{0,1,\ldots, n-1\}$; 
e.g., $1=\{\varnothing\}$, $2=\{\varnothing, \{\varnothing\}\}$.
Then each natural number is both an element and a (finite) 
subset of natural numbers.
Next, let $V_\bomega=\{a,b,\ldots\}$ be the collection of 
all finite subsets of the natural numbers, a countable
set that strictly includes $\bomega$.
This $V_\bomega=\{a,b,\ldots\}$ is a model of the von Neumann universe,
the class of hereditarily finite sets \cite{ENDERTON}.
For $a, b \in V_\bomega$, write $a \sim b$ if 
$a \in b$ or $b \in a$.
By the foundation (or regularity) axiom of ZFC \cite{ENDERTON} 
only one of these memberships can be true.
Considering $\sim$ as a set of edges, the graph $(V_\bomega, \sim)$ is
isomorphic to $\Rado$.
%therefore we can, if we wish, naturally orient each edge.)

A second representation of $\Rado$ is obtained by taking $\bomega$ to be a set
of vertices instead of $V_\bomega$.
%(call it the $\bomega$ graph) and the following edge relationship. 
To define edges, we first represent each natural number
$n$ in binary 
%by writing $n = \sum_{m \ge 0} b_m 2^m$, where $b_m=0$ or $1$ for all $m$.
%Equivalently, 
by writing $n=\sum_{m \in B(n)} 2^m$.
%where $B(n) := \{m \ge 0: b_m =1\}$.
Write $m \sim' n$ 
%We then declare that there is an edge between two elements $m$ and $n$ of $\bomega$
if $m \in B(n)$ or $n \in B(m)$.
Considering $\sim'$ as a set of edges, we have that $(\bomega, \sim')$ is
isomorphic to $(V_\bomega, \sim)$ and hence to $\Rado$.

The isomorphism between the two graphs, 
$(V_\bomega, \sim)$ and $(\bomega, \sim')$,
is via a function 
$A : V_\bomega \to \bomega$, devised by Ackermann \cite{ACK37}
via the recursive formula
$A(a) = \sum_{b \in a} 2^{A(b)}$, for $a \in V_\bomega$.
It is easy to see that (i) $A$ is a bijection and (ii) $A$ preserves edges.
If we view $B$ as a function $\bomega \to V_\bomega$ that assigns the
set $B(n)$ to the natural number $n$ then we can see that $B$ is the inverse
of $A$.

Some properties of $\Rado$ are as follows.

First, if we partition the set $V(\Rado)$ of vertices of $\Rado$ 
into finitely many sets then 
$\Rado$ is isomorphic to the induced subgraph on one of these parts.
For example, considering $\bomega$ as a subset of $V_\bomega$,
we have that the $V_\bomega$ graph is isomorphic to the induced subgraph
on $V_\bomega \setminus \bomega$ (but not to the induced subgraph
on $\bomega$ because this subgraph is complete: for every two distinct natural
numbers $m, n$ either $m \in n$ or $n \in m$).

Second, every finite or countably infinite graph can be embedded
as an induced subgraph of $\Rado$.
Thus, $\Rado$ contains every possible graph!

The key property of $\Rado=(V(\Rado), E(\Rado))$ is the following. Given finitely
many distinct vertices $u_1, \ldots, u_m, v_1,\ldots, v_n \in V(\Rado)$,
there is $z \in E(\Rado)$ such that $z$ is adjacent to $u_1, \ldots, u_m$
and nonadjacent to $v_1, \ldots, v_n$.
In fact, any graph satisfying this key property is isomorphic to $\Rado$;
see \cite{CAM97}.
It is easy to see that for any countably infinite sets $U_1, U_2$
and any $0<p_1, p_2<1$, the graphs $G(U_i, p_i)$ satisfy this key property and  
so, if $X_1, X_2$ are two random graphs, independent or not,
with laws $G(U_1, p_1)$, $G(U_2, p_2)$, respectively,
then $P(X_1 \cong X_2 \cong \Rado)=1$. The symbol $\cong$ stands for
``being isomorphic to''; see Def.\ \ref{isodef}.

As noted by Chatterjee and Diaconis \cite{ChD23}, this poses a conundrum.
On one hand, if $X, Y$ are independent $G(\infty,1/2)$ random graphs
then $X \cong Y$ a.s.
On the other hand, if  $X, Y$ are independent $G(n,1/2)$ random graphs
then the probability that they 
are isomorphic is astronomically small as $n \to \infty$, 
namely, at most $n!/2^{\binom{n}{2}} \le e^{-c n^2}$, for some 
positive constant $c$.
%A key property of the Rado graph is that if we  partition $V$ into two
%components 
%then the Rado graph is isomorphic to its restriction on the infinite component.
%Moreover, this is a charactererizing property. It is, of course, easy to see
%that this property holds for $G(\N, p)$.
%Moreover, every finite graph appears as an isomorphic copy of a subgraph
%of the Rado graph.
%As deduced from the above the countable Erd\H{o}s-R\'enyi graph
%has the following property. If $X, Y$ are two countable Erd\H{o}s-R\'enyi graphs,
%with possibly different (but nontrivial, i.e., neither $0$ nor $1$)
%edge probabilities and possibly dependent then $X$ is isomorphic to $Y$ a.s.
To shed light to this,
they asked the question whether a $G(m,1/2)$
graph can be embedded into an independent $G(n,1/2)$ graph when $m$ is smaller
than $n$ and found that there is a critical value of $m$,
roughly of order $2\log_2 n$ (but the exact value of it
is very important), such that the probability that the
small graph can be embedded into the larger one tends to $1$ or $0$
depending on whether $m$ is  below or above the threshold.
For a precise statement of this result see \cite[Thm.\ 1.2]{ChD23}
or the more general Theorem \ref{thm1} below.

In this paper, we consider two problems that we call
{``graph embedding'' and ``common subgraph'' problems.}
The first one refers to the question whether a $G(m,p)$ graph
can appear as an {\em induced} subgraph of an independent $G(n,q)$ graph,
when $m \le n$.
The second one refers to the question whether two independent 
$G(n,p)$, $G(n,q)$ graphs contain a common {\em induced} subgraph
of size $m$. 
The terminology is not standard. For example, the ``graph embedding'' 
problem is often referred to as ``subgraph isomorphism'' and is a classical
problem in computer science; see, e.g., Ullmann \cite{ULL76}.
We note that both problems have been treated by Chatterjee and Diaconis
\cite{ChD23} in the uniform distribution case (random graph of size $n$
refers to a uniform probability measure on the set of all graphs on
$n$ vertices). Deviation from the non-uniform case is a harder problem and
it is what we are interested in in this paper.
The main theorems are  Theorem \ref{thm1} (graph embedding)
and Theorem \ref{thm2} (common subgraph).
In both cases we show the existence of a phase transition
occurring at two integers that differ by at most 2.
This implies that we have what is sometimes known as a
{\em two-point concentration} phenomenon. 
This means that phase transition occurs at a ``boundary''
defined by two integer sequences, $m_-(n)$, $m_+(n)$, such
that, eventually, $m_+(n) - m_-(n)$ is either $1$ or $2$.
We show that there is flexibility and that the set of integers $n$
such that $m_+(n) - m_-(n) > 1$ can be made as small as possible
by choosing a certain sequence that tends to infinity as slowly as possible
(this is the sequence $C_n$ appearing in our two main theorems).
However, the set of $n$ such that $m_+(n) - m_-(n) > 1$ can 
never become empty.

%For example,
%we prove that there is a $m_* \equiv m_*(n, p, q)$ such that the maximum
%common induced subgraph between two independent $G(n, p)$, $G(n, q)$
%has at least $m_-(n):=\lfloor m_* - \frac{C_n}{\log n} \rfloor$ and
%less than $m_+(n):=\lceil m_* + \frac{C_n}{\log n} \rceil$ vertices
%with probability that tends to $1$ as $n$ increases to infinity,
%where $C_n$ is an arbitrary sequence that tends to infinity and
%is $o(\log n)$.

%{\red
%We pay particular emphasis to the thresholds for both problems                    
%and identify it in (in some sense) the best possible way:
%we can make the number of integers $n$ for which the difference
%is $2$ as small as possible by choosing a certain sequence,
%called $C_n$, that tends to infinity as slowly as possible.}
%

Phase transition phenomena in computational problems are closely related to
the complexity of solving these problems and provide significant insights into
the difficulty of certain instances compared to others. In the context of
computational complexity, a phase transition refers to the abrupt change in the
solvability or structure of a problem as some parameter is varied. This concept
is often studied in NP-complete problems like k-SAT \cite{COP2016} or
graph coloring \cite{AN2005}, where, as parameters                                
(e.g., the ratio of constraints to variables) cross a critical
threshold, the likelihood of finding a solution shifts dramatically—from being
almost always solvable to being almost always unsolvable.

We first heard about these problems during a talk given by
Persi Diaconis \cite{OWPS}. Subsequently, Chatterjee and Diaconis,
provided a phase transition for the case of uniform graphs in \cite{ChD23},
a paper that provided further motivation for us.
While pursuing open problems stated in \cite{ChD23},
two papers appeared on the ArXiV:
Lenoir \cite{L2024}, who studied phase transitions for
uniform hypergraphs, and 
Surya, Warnke, and Zhu \cite{SWZ2023}, who studied
the phase transition phenomena of interest to us. 
The second paper deals with precisely the same problem as ours in a more general
context (that is, without restrictions on the parameters $p$ and $q$ of the
random graphs.)
%Although \cite{SWZ2023}
%reports plausible results that appear more general than ours,
%their proofs seem to contain several lacunae and rely on heuristics that are
%not easily justifiable.
In view of this, our paper can only be seen as dealing with the problem under
restrictions on the parameters (see Theorems \ref{thm1} and \ref{thm2} below).
As such, the only new thing in our paper is that it gives a different, essentially
combinatorial proof, for a special case,
exhausting the limits of the second moment method.
%without conditioning on events of high probability.
On the other hand, Surya, Warnke, and Zhu use a clever probabilistic idea
in order to deal with the general case.

Both problems are also of interest in several applications as well.
In artificial intelligence research,
people are interested in discovering whether a pattern occurs
inside a large target graph. McCreesh {\em et al.} \cite{AIpaper18}
have studied this numerically and predicted the 
existence of a phase transition. 
In bioinformatics research, graphs represent biological networks at
the molecular or higher (protein or even species) level.
Again, an important question is that of locating a specific
pattern in a network. An algorithm for the problem, is proposed by
Bonnici {\em et al.} \cite{BIOpaper13}.
The problems are also of interest in theoretical computer science
since subraph isomorphism problems are related to constraint
satisfaction problems.

\section{\bfseries\scshape  Main results}
Before stating the results, we introduce some notation.
Let $V$ be a set and denote by $\PP(V)$ the collection of all its subsets,
and by $\PP_2(V)$ the collection of subsets of $V$ of size $2$.
%Let $\PP_k(V)$ be the collection of subsets of $V$ of size $k$. 
Any $E \subset \PP_2(V)$ defines a graph $\Gamma=(V, E)$.
Equivalently, we can think of 
$\Gamma$ as being a collection of $\{0,1\}$-valued numbers
\[
\text{$X(e)$, where $e$ ranges over $\PP_2(V)$,}
\]
because $E$ can be identified with the set 
\[
\{e\in \PP_2(V):\, X(e)=1\}.
\]
%Thus, if $\Gamma$ has law $G(V,p)$ then $\Gamma_{u,v}$, $u, v \in V$, $u \neq
%v$, are i.i.d.\ Bernoulli$(p)$ random variables.
An element $e$ of $\PP_2(V)$ may be called an edge of $V$
(having in mind the complete graph).
The two elements of $e$ are called {\em endpoints} of $e$.
We can also think of $E$ as a binary relation on $V$ that a priori possesses
no properties other than symmetry.

If $A \subset V$ then $\Gamma^A$ will denote the 
the {\em induced subgraph} (or, simply, {\em restriction})  of $\Gamma$ on $A$
whose edges are all edges of $\Gamma$ with endpoints in $A$.
%Any function $f$ with domain $V$ immediately induces a function on $E$
%by the rule $f(e) = \{f(u), u \in e\}$.

Let $V$, $V'$ be two sets. %$\Gamma'=(V', E')$ be two graphs. 
Any function 
\[
f: V \to V'
\]
defines a function 
\[
\widetilde f: \PP(V) \to \PP(V')
\]
by mapping any $e \in \PP(V)$
to the set $\widetilde f(e) = \{f(x): x \in e\}$.
We will omit the tilde over $f$ when no ambiguity arises.
If $\Gamma=(V, E)$ is a graph and $f: V \to V'$ a function
then, letting $f(E) := \{f(e): e \in E\}$, the
object $f(\Gamma)= (f(V), f(E))$ is a graph provided that $f(e)$ has
cardinality 2 for all $e \in E$.
%We will use the same letter $f$ for the restriction of $\widetilde f$ to $E$
%and denote by $f(u)$ the image of $u \in U$ by the
%injection $f$ and $f(e)=\{f(u),f(v)\}$ the image of
%$e \in \{u,v\} \in E$.

Since the notion of isomorphism is central to this paper,
we recall its standard definition below.
\begin{definition}[isomorphism]
\label{isodef}
Let $\Gamma = (V, E)$, $\Gamma' = (V',E')$ be two graphs.
We say that $f: V \to V'$ is an {\em isomorphism} if $f$ is a bijection
and if $e \in E \iff f(e) \in E'$.
In this case, $V'=f(V)$, $E'=f(E)$, and $\Gamma'=f(\Gamma)$.
The statement that there exists an isomorphism $f$ between $\Gamma$ and
$\Gamma'$ is abbreviated as $\Gamma \cong \Gamma'$.
\end{definition}

An isomorphism between $\Gamma$ and itself is called automorphism.
%If $\Gamma'=\Gamma$ then $f$ is called an automorphism. 
The set of automorphisms forms a group and the larger its size the more
symmetric $\Gamma$ is. For example the empty graph (no edges)
and the full graph are fully symmetric. The set of automorphisms of the
$\Rado$ is quite remarkable and has been studied by Truss \cite{TR85}.

%Let $m, n$ be positive integers. We write $[n]$ for the set $\{1,\ldots,n\}$
%and let $\II_{m,n}$ be the set of injections from $[m]$ into $[n]$,
%a set of size $(n)_m := n (n-1)\ldots(n-m+1)$.

In the first part of this paper we deal with an embeddability problem.
We give the definition of the term below.
\begin{definition}[embeddability]
We say that $\Gamma$ is {\em embeddable} in $\Gamma'$ if $\Gamma \cong 
\Gamma''$ for some induced subgraph $\Gamma''$ of $\Gamma'$.
\end{definition}

Our first concern is whether an Erd\H{o}s-R\'enyi random graph
can be embedded in a bigger independent random graph.
To this end, we have the following theorem.

\begin{theorem}[phase transition for the graph embedding problem]
\label{thm1}
Let $X_m, Y_n$ be two independent random graphs with laws 
$G(m, p), G(n, \frac{1}{2})$, respectively, where $0<p<1$.
Let $m(n)$ be a sequence of positive integers such that $m(n) \to \infty$.
Then
\begin{numcases}
{\lim_{n \to \infty} \P(X_{m(n)} \text{ is embeddable in } Y_n)=}
1, 
& $m(n) = \left\lfloor 2 \log_2 n + 1 - \frac{C_n}{\log n} \right\rfloor=: m_-(n)$
\tag{I}
\\
0, 
& $m(n) = \left\lceil 2 \log_2 n + 1 + \frac{C_n}{\log n} \right\rceil=: m_+(n)$ 
\tag{II}
\end{numcases}
where $C_n \to \infty$ and $C_n/\log n \to 0$.
\end{theorem}
\begin{remark}
\label{Wrem}
\begin{enumerate}[(i)]
\item
We will use the term ``{\em phase I}'' for the case when 
$X_{m(n)}$ is embeddable in $Y_n$ with probability tending to $1$.
Similarly, we will call ``{\em phase II}'' the case when the same
event has probability tending to $0$.
\item We should actually read the last sentence as: $C_n \to \infty$ arbitrarily
slowly. 
\item
The difference $m_+(n)-m_-(n)$ between the two integers 
at the right-hand sides of (I) and (II)
is either $1$ or $2$ for $n$ large enough.
The set $\Delta_* := \{n \in \N:\, m_+(n) - m_-(n) > 1\}$ 
contains all powers of $2$.
\item
In some sense, the gap 
%between the two integers at the 
%right-hand sides of (I) and (II) 
is as small as possible.
If we consider the number $N$ of embeddings of $G(m(n),p)$ into
$G(n,1/2)$ and ask for which $m=m(n)$ we have 
$\lim_{n \to\infty}\E N =\infty$, or $\lim_{n\to \infty}\E N = 0$, respectively,
then we can see that 
$m(n) \le m_-(n)$,
or
$m(n) \ge m_+(n)$,
respectively.
%Moreover, the set of $n$ such that $m_+(n)-m_-(n)=2$ is,
%in some sense, least.
Moreover, the slower $C_n$ converges to $\infty$,
the smaller the set $\Delta_*$ is.
%Moreover, if $C_n/\log n$, with $C_n \to \infty$ arbitrarily slowly,
%were to be replaced by a sequence
%that was asymptotically bigger (e.g., $\log\log n/\log n$)
%then that would increase the
%set of $n$ such that $m_+(n)-m_-(n)=2$.
We have a freedom to choose $C_n$ and this freedom is part of the theorem.
\item
Clearly, the conditions for phase I and II are almost,
but not exactly, complementary. In fact, when 
both conditions are violated, 
several things can happen, depending on the
precise choice of the integer sequence $m(n)$.
For example, if $n_k = 2^k$ and $m_k=2k+1$ then
the behavior of graph embedding is a very delicate question.
\item 
It is easy to see that 
\[
\P(X_m \text{ is embeddable in } Y_n)
\ge \P(X_{m+1} \text{ is embeddable in } Y_n).
\]
Indeed, if we think of $X_{m+1}$ as a graph on the set of vertices
$\{1,\ldots,m+1\}$ and if $X_{m+1}$ is embeddable in $Y_n$ then
the restriction of $X_{m+1}$ on $\{1,\ldots,m\}$ is
also embeddable in $Y_n$. But this restriction has the same
law as $X_m$.
\item
We see that the effect of $p$ vanishes from
the conditions, and the statement remains identical
to the one corresponding to the case $p=1/2$, this case being the one 
treated in \cite{ChD23}. % under slightly more restrictive conditions.
This fact is easy to see insofar as phase II is concerned 
(see Section \ref{PH2emb} below) but it is far less
trivial for the other phase (treated in Section \ref{PH1emb}.)
\end{enumerate}
\end{remark}

The above result as well as the one below address some of the
open problems stated in \cite{ChD23}. In particular, \cite{ChD23} states
that, since all $G(\infty, p)$ graphs are isomorphic to $\Rado$, for
any $0<p<1$, understanding the largest isomorphic induced subgraph
of two independent $G(n,p)$ graphs is an interesting problem.
We shall consider two independent such graphs, $G(n,p)$, $G(n,q)$,
with $p$ not necessarily equal to $q$ and look for conditions
that establish the existence of a common subgraph of largest possible size.
To be precise, we give the following definition.

\begin{definition}[$m$-isomorphism]
\label{misodef}
Graphs $\Gamma, \Gamma'$ (of possibly different sizes)  
are {\em $m$-isomorphic}, and denote this by
\[
\Gamma \jso{m} \Gamma',
\]
if there are 
induced subgraphs $\Delta, \Delta'$ of $\Gamma, \Gamma'$, respectively,
both of sizes $m$, such that $\Delta \cong \Delta'$.
A bijection from $\Delta$ to $\Delta'$ will be called an 
{\em $m$-isomorphism}.
\end{definition}

The word ``$m$-isomorphic'' is ours and perhaps, linguistically,
not the best choice. The idea is that two $m$-isomorphic graphs
$\Gamma$ and $\Gamma'$ have a size-$m$ part that looks the same.
The larger the $m$ the more similar the graphs are. We are interested
in the largest such $m$.
This is what we refer to as ``{\em common subgraph problem}''.
In particular, if $\Gamma,\Gamma'$ both have sizes $n$, then they
are $n$-isomorphic iff they are exactly isomorphic.

Our next goal is to consider the common subgraph
problem between $G(n,p)$ and $G(n,q)$ when
$p,q$ are strictly between $0$ and $1$. Theorem \ref{thm2} is stated in terms of
the quantities
\begin{equation}
\label{taujk}
\begin{split}
\tau_{j,k} &\equiv \tau_{j,k}(p,q) := p^jq^k + (1-p)^j (1-q)^k,
\\
\tau &:= \tau_{1,1},
\\
\lambda &:= 1/\log(1/\tau),
\end{split}
\end{equation}
and the functions below. First let
\begin{equation}
\label{Wdef}
W(x) := x+ 2\lambda \log x + \frac{\lambda}{x} \log(2\pi x), \quad x \ge 1.
\end{equation}
It is easily seen that $W$ is strictly increasing 
and strictly concave
with $\lim_{x \to \infty} W(x) = \infty$
and $\lim_{x \to \infty} W'(x)=1$. 
See Section \ref{suppinfo}.
We then define $m_*(n)$ via
\begin{equation}
\label{mstar}
W(m_*(n)) = 4 \lambda \log n + 2 \lambda+1.
\end{equation}

\begin{figure}
\begin{center}
\includegraphics[width=0.35\textwidth]{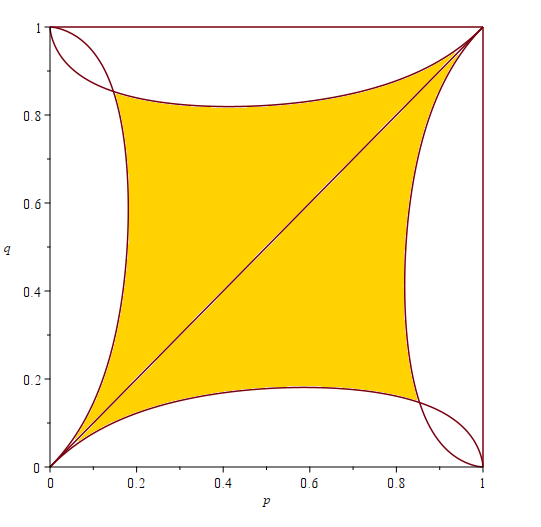}
\captionof{figure}{The set $\Y$ of $(p,q)$ for which 
we have a phase transition in the common subgraph problem. We refer to
$\Y$ as the admissible region.}
\label{figpq}
\end{center}
\end{figure}

\begin{theorem}[phase transition for the common subgraph problem]
\label{thm2}
Define
\begin{equation}
\label{thespis}
\Y:=\left\{(p,q) \in (0,1)\times (0,1):\, \max\{\tau_{1,2}(p,q), \tau_{2,1}(p,q)\right\} 
< \tau(p,q)^{3/2}\},
\end{equation}
a region depicted in Figure \ref{figpq}.
Let $X_n, Y_n$ be two independent random graphs with laws $G(n, p), G(n,q)$,
respectively, and with $(p,q) \in \Y$.
%\max(\tau_{1,2}, \tau_{2,1}) < \tau^{3/2},
Then
\begin{numcases}                                                                  
{\lim_{n \to \infty} \P\left(X_n \stackrel{m(n)}{\cong} Y_n\right) = }
1, & $m(n) = \lfloor m_*(n)-(C_n/\log n)\rfloor$ \tag{I}
\\
0, & $m(n)= \lceil m_*(n)+(C_n/\log n)\rceil$ \tag{II}
\end{numcases}
where $C_n \to \infty$ such that $C_n/\log n \to 0$.
\end{theorem}

\iffalse
\begin{remark}
(i) [crude asymptotics]
\label{Wrem}
It is easily seen that $W'(x) \ge 1$ for all $x$ and that
$m_*(n) \sim 4 \lambda \log n$ as $n \to \infty$.
\\
(ii)
[fine asymptotics]
{\blue Before:
We define the function $W(x)$, $x>0$, as the unique solution $y=W(x)$
to the equation $y e^y =x$. This is the so-called Lambert $W$ function 
(due to Johann Heinrich Lambert  \cite{L1758}).
It is easy to see that $W(0)=0$, $W(x) \sim x$ as $x \to 0$ and
$W(x) \sim \log x$ as $x \to \infty$.
The asymptotic behavior of $W(x)$ for $x$ large is better captured
via the inequality
\[
\left|W(x) -  (\log x - \log \log x)\right| \le 2 \frac{\log\log x}{\log x},
\]
easily deduced by the asymptotic expansion \cite[(4.18)]{C1996}.
The implicit equation for $m_*(n)$ is
of the form $z+ a \log z = b$ and this is easily solved in terms 
of the Lambert W function via $z=a W(a^{-1} e^{b/a})$.
These two observations result in nice asymptotics for $m_*(n)$ as $n \to \infty$.
}
\tkmid{5cm}{Find asymptotics for $m_*(n)$.}
$\vdots$
\tkmid{0.9\textwidth}{Dimitris is supposed to be working on further items for
this section.}
\end{remark}
\fi
\begin{remark}
\begin{enumerate}[(i)]
\item
The points made in Remark \ref{Wrem}(i)--(v), with obvious
modifications, remain valid for theorem \ref{thm2} also.
\item
The sequence $m_*(n)$ satisfies
\[
m_*(n) = 4\lambda \log n+2\lambda+1
-2\lambda \log(4\lambda \log n+2\lambda+1) 
+ O\left(\frac{\log\log n}{\log n}\right),
\]
as $n \to \infty$.
See Lemma \ref{mstarapprox} in Section \ref{suppinfo}.
We can verify, numerically, that the approximation is extremely sharp
when $(p,q)$ is away from $(1,1)$ or $(0,0)$.
\end{enumerate}
\end{remark}

\section{\bfseries\scshape The graph embedding problem}
Fix two finite sets $U, V$ with cardinalities $m, n$ respectively, where $m \le n$.
Elements of $\PP_2(U)$ are called edges of $U$. Similarly, for $\PP_2(V)$.
Consider two independent random graphs $X\equiv X_m, Y\equiv Y_n$ such that
$X_m$ has law $G(U,p)$ and $Y_n$ has law $G(V,1/2)$.
Identifying $X$ with a collection of i.i.d.\ Bernoulli$(p)$ random variables
$\{X(e), e \in \PP_2(U)\}$, we refer to those
$e \in \PP_2(U)$ such that $X(e)=1$ as edges of $X$.
Similarly for $Y$.
Let
\[
\II\equiv \II_{U,V} :=\text{ the collection of injective functions
from $U$ to $V$,}
\]
a set of size $(n)_m = n(n-1)\cdots(n-m+1)$.
If $f \in \II_{U,V}$ we let $\ran f \subset V$ be 
its range. Then
\begin{equation} \label{N1}
N = \sum_{f \in \II_{U,V}} \I_{f(X) = Y^{\ran f}}
\end{equation}
is the number of isomorphisms between the two random graphs, since, by 
definition, $Y^{\ran f}$ is the restriction of $Y$ onto $\ran f$.
The event of interest is
\[
\{X \text{ is embeddable in } Y\} = \{N > 0\}.
\]
%\[
%\P(\exists A \subset V \, X \cong Y^A) =
%\P(\exists f \in \II_{U,V} \, fX = Y^{R(f)}) =  \P(N > 0).
%\]
\subsection{\bfseries Phase II of the graph embedding problem}
\label{PH2emb}
Since $m$ will be taken to be much smaller than $n$, it is not unreasonable
to postulate, in view of 
\[
\P(N > 0) \leq \E N,
\]
that the the threshold for $m$ will be the ``least'' function of $n$ such that
$\E N \to 0$.
%Recalling that, for $e \in \PP_2(U)$,  we have $X(e)=1$ iff $e$
%is an edge of the graph $X$, we have
\begin{align}
\E N &= \sum_{f \in \II_{U,V}} \P(\forall e \in \PP_2(U)~ X(e)=Y(f(e)))
\nonumber
\\
&= \sum_{f \in \II_{U,V}} \prod_{e \in \PP_2(U)} \P(X(e)=Y(f(e)))
\nonumber
\\
&= (n)_m \left(\tfrac12 p + \tfrac12 (1-p) \right)^{\binom{m}{2}}
= (n)_m 2^{-\binom{m}{2}}.
\label{ENone}
\end{align}
We can easily see that
\begin{equation}
\label{stitay}
\text{$(n)_m/n^m \to 1$ as $n \to \infty$ when $m=m(n)=O(\log n)$}.
\end{equation}
The question then becomes that of finding the least $m=m(n)$ such
that 
\[
n^m 2^{-\binom{m}{2}} \to 0.
\]
\begin{lemma}
If $m=m(n) = \left\lceil 2 \log_2 n + 1 + \frac{C_n}{\log n}\right\rceil$
and $C_n \to \infty$, then $\E N \to 0$.
\end{lemma}
\begin{proof}
By \eqref{stitay}, we simply need to show that $n^m 2^{-\binom{m}{2}}\to 0$.
We have 
\[
n^m 2^{-\binom{m}{2}} 
= 2^{- \frac12 m(m-2 \log_2 n -1)}.
\]
But
\[
m(m- 2 \log_2 n -1) \ge m \frac{C_n}{\log n} \to \infty,
\]
because $m/\log_2 n \to 2$ and $C_n \to \infty$.
\end{proof}

This proves the second part (phase II) of Theorem \ref{thm1} but also gives us
reasons to suspect that the ``cut-off function'' 
$2\log_2 n+1$ will work for the phase I too.
We remark that, even though this function does not depend on $p$ (and this
is because $G(n,1/2)$ remains the same in law if we swap edges and non-edges),
this does not make the other part (phase I) of the theorem trivial 
when $p\neq 1/2$. We need to work harder to show that the effect of $p$ vanishes
as will be seen below.

\subsection{\bfseries Phase I of the graph embedding problem}
The rest of the proof proceeds on the basis of the inequality
\begin{equation}
\label{CS}
\P(N>0) \ge \frac{(\E N)^2}{\E N^2}.
\end{equation}
The goal is to show that $\P(N>0) \to 1$ under the conditions for phase I
of Theorem \ref{thm1}. %restated here:. 
Throughout the rest of this section, we let
\[
%m_*(n) = 2 \log_2 n+1, \quad
%m(n)-m_*(n)\to 0, \quad 
m \equiv m(n) = \left\lfloor 2 \log_2 n+1 - {C_n}/{\log n} \right\rfloor,
\]
for some sequence $C_n \to \infty$ with $C_n/\log n \to 0$.
Our plan consists in finding a suitable upper bound for the
reciprocal of the right-hand side of \eqref{CS}, say,
\[
\E N^2/(\E N)^2 \le \SSS;
\]
see Proposition \ref{esoteric} below;
and then showing that
\[
\limsup_{n \to \infty} \SSS \le 1
\]
see Proposition \ref{S1part}. This will 
conclude the proof of phase II of \eqref{thm1}.

To achieve this, we need to use an auxiliary edge graph, 
i.e., a graph whose Vertices are elements of $\PP_2(U ) \cup \PP_2(V)$.
This is motivated by the trivial relation \eqref{quintix} in Section \ref{PH1emb}.
Such a device will facilitate the computation of correlations.
Subsequently, we make some combinatorial estimates of various sets of
pairs of injections from $U$ to $V$ (Section \ref{CoCo}),
and then estimate the sizes of various classes of components of
the auxiliary edge graph (Section \ref{SiSi}).
We finally put things together to obtain a good tight bound $\SSS$
in Section \ref{SSSsubsec}.

%An edge graph is a graph whose vertices are edges of $U$ or $V$,
%that is, elements of $\PP_2(U) \cup \PP_2(V)$.

\subsubsection{\bfseries The auxiliary edge graph}
\label{PH1emb}
First, some notational convention.
A priori, $f \in \II_{U,V}$ acts on points $u \in U$.
At a higher level, it also acts naturally on subsets of $U$, and, in particular,
on edges $e \in \PP_2(U)$. We shall use the notation 
\[
\PP_2(U) \ni e \mapsto f(e) \in \PP_2(V)
\]
for this action. At an even higher level, it acts on collections of subsets of $U$,
so $f(\PP_2(U))$ denotes the set $\{f(e): e \in \PP_2(U)\}$.

As before, let ($X(e)$, $e \in \PP_2(U)$) be i.i.d.\ Bernoulli$(p)$ random
variables and let ($Y(e)$, $e \in \PP_2(V)$) be i.i.d.\ Bernoulli$(1/2)$ random
variables, the two collections being independent.
%The expression for $N$ in \eqref{N1} can be written as
%\[
%N= \sum_{f\in \red \II_{U,V}} \prod_{e \in \PP_2(U)} \1_{X(e)=Y(f(e))}.
%\]
%Squaring this gives
From \eqref{N1} we get
\begin{equation}
\label{quintix}
N^2 =  \sum_{f, g \in \II_{U,V}} \prod_{e \in \PP_2(U)} \1_{X(e)=Y(f(e))=Y(g(e))}.
\end{equation}
We use the following device to rewrite the condition in the last indicator 
function.
\begin{definition}
\label{croesus}
Let $\EE(f,g)$ be the collection of pairs
$\ee = (e, e')$ where $e \in \PP_2(U)$ and $e'=f(e)$ or
$e'=g(e)$, that is,
\[
\EE(f,g):= \big\{ (e, f(e)):\, e \in \PP_2(U)\big\}
\cup \big\{ (e, g(e)):\, e \in \PP_2(U)\big\}.
\]
If $\ee = (e, e') \in \EE(f,g)$, 
we write
\[
X=Y \text{ on } \ee \overset{\text{def}}{\iff} X(e)=Y(e').
\]
\end{definition}
We then have
\[
\forall e \in \PP_2(U)~~ X(e) = Y(f(e))=Y(g(e)) 
\iff 
\forall \ee \in \EE(f,g)~~  X=Y \text{ on } \ee,
\]
and so \eqref{quintix} reads
\begin{equation}
\label{sextix}
N^2 = \sum_{f,g \in \II_{U,V}} 
\prod_{\eee \in \EE(f,g)} \1_{X=Y \text{ on } \eee}.
\end{equation}

It is useful to rearrange the product in the last display into
a product of independent random variables. To do this,
we think of every $\ee \in \EE(f,g)$ as an Edge
in a graph whose set of Vertices is the set 
$\PP_2(U) \cup f(\PP_2(U)) \cup g(\PP_2(U))$.
Denote this graph by $\TT(f,g)$ and note that it is bipartite
since every $\ee \in \EE(f,g)$ has one endpoint in 
\begin{equation}
\label{LeftV}
\LL:=\PP_2(U) \quad \text{ (the set of left Vertices)}
\end{equation}
and another endpoint in 
\begin{equation}
\label{RightV}
\RR(f,g):=f(\PP_2(U)) \cup g(\PP_2(U)) \quad \text{ (the set of right Vertices)}.
\end{equation}
That is,
\begin{equation}
\label{Tgraph}
\TT(f,g) := \big( \LL \cup \RR(f,g), \, \EE(f,g)\big).
\end{equation}
\begin{ling}
%We adopt the following linguistic convention:
We capitalize the first letter in the word Vertices to emphasize the fact
that a Vertex of $\TT(f,g)$ is a actually an edge of $U$ or $V$. 
Similarly, we write Edge to emphasize the fact that it is a pair
of edges.
\end{ling}
Note that although the Edges $\ee$, that is the elements of $\EE(f,g)$
as in Def.\ \ref{croesus}, are directed, we assume that they lose their
direction when we define $\TT(f,g)$.
See Figure \ref{figtfg} for an example of $\TT(f,g)$ when                         
$m=4$, $n=6$ and two specific injections $f, g: [m]\to [n]$.

\begin{figure}
\begin{center}
\includegraphics[width=0.60\textwidth]{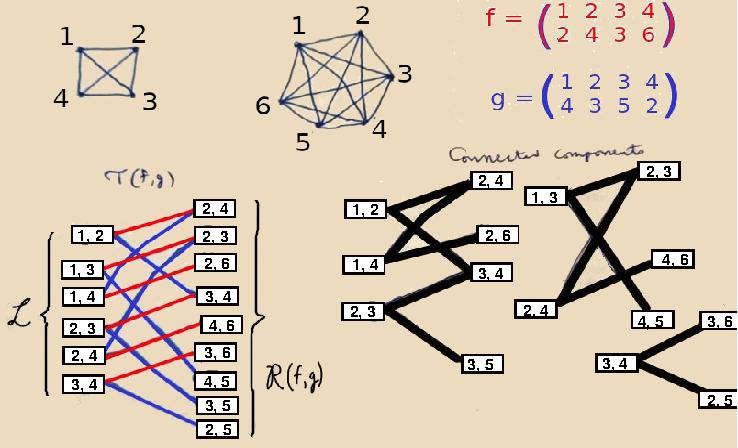}
\captionof{figure}{Assume that $m=4$, $n=6$,
and $f, g$ the injections shown in the figure.
Then $\TT(f,g)$ is shown on the bottom left and its
connected components on the bottom right.} 
\label{figtfg}
\end{center}
\end{figure}

A {\em connected component} $C$ of $\TT(f,g)$ is a connected subgraph.
Write
\[
\CC(f,g) := \text{ set of connected components of $\TT(f,g)$}.
\]
\begin{definition}
If $C \in \CC(f,g)$ we write
\[
X=Y \text{ on } C \iff X=Y \text{ on } \ee \text{ for all Edges } \ee \text{ of } C.
\]
\end{definition}
Since, by definition, two connected components share no common edges,
\eqref{sextix} is written as
\[
N^2 = \sum_{f,g \in \II_{U,V}} \prod_{C \in \CC(f,g)} \1_{X=Y \text{ on } C}.
\]
But now note that the last product is a product of independent random variables
so that
\begin{equation}
\label{septix}
\E N^2 = \sum_{f,g \in \II_{U,V}} \prod_{C \in \CC(f,g)} \P(X=Y \text{ on } C).
\end{equation}

\begin{lemma}
Given $f, g \in \II_{U,V}$,
if $C$ is a connected component of $\TT(f,g)$
has $j$ Vertices in $\LL$ and $k$ in $\RR(f,g)$ then
\begin{equation}
\label{taujkuno}
\P(X=Y \text{ on } C) = (p^j+(1-p)^j) (1/2)^k =: \tau_{j,k}.
\end{equation}
\end{lemma}
\begin{proof}
Let $C$ be a connected component of $\TT(f,g)$ and let $\text{Vert}(C)$ be
the set of its Vertices.
By definition, $X=Y$ on $C$ is equivalent to $X(e)=Y(e')$ for
all Edges $(e, e')$ of $C$, where $e \in \LL$, $e'\in \RR(f,g)$
(since $C$ is bipartite). 
Since $C$ is connected, $X=Y$ on $C$ then means that the random
variables $(X(e), e \in \LL\cap \text{Vert}(C)) 
\cup (Y(e'), e' \in \RR(f,g)\cap\text{Vert}(C))$
are all equal.
Since $\LL\cap \text{Vert}(C)$ has size $j$ and $\RR(f,g)\cap\text{Vert}(C)$
has size $k$, we have $j+k$ independent Bernoulli random variables,
with $j$ of them having parameter $p$ and the rest $1/2$.
The probability that they are all equal to $1$ is $p^j (1/2)^k$
and the probability that they are all equal to $0$ is
$(1-p)^j (1/2)^k$. So \eqref{taujkuno} is proved.
\end{proof}

So if we let
\begin{equation}
\label{compcl}
\CC_{j,k}(f,g) := \{C \in \CC(f,g):\, C \text{ has $j$ Vertices in $\LL$ and $k$ in $\RR(f,g)$}\}
\end{equation}
\eqref{septix} becomes
\begin{equation}
\label{N2bis}
\E N^2 
= \sum_{f,g \in \II_{U,V}} \prod_{j, k \ge 1} \tau_{j,k}^{|\CC_{j,k}(f,g)|},
\end{equation}
from which it is evident that we need to obtain information about
$\CC_{j,k}(f,g)$ and their sizes.
We first observe that any the number of right Vertices of a connected component
minus the number of left Vertices is either $0$ or $1$:
\begin{lemma}
\label{classlem}
$\CC_{j,k}(f,g) \neq\varnothing$ iff $k=j$ or $k=j+1$.
\end{lemma}
\begin{proof}
Let $e \in \LL$. If $f(e)=g(e)=e'$ then $\{e,e'\} \in \EE(f,g)$.
Since $f, g$ are injections the only preimage of $e'$ under $f$ and under $g$ is
$e$. Hence $\{e,e'\}$ is a connected component, an element of $\CC_{1,1}(f,g)$.
%\tkside{We need to decide whether to write $(e,e')$ or $\{e,e'\}$ for
%the elements of $\EE(f,g)$.}
Let $C$ be a connected component not of this type. Assume $C \in \CC_{j,k}(f,g)$.
%Think of EDGES of the form  $\{e,f(e)\}$
%as red and EDGES of the form $\{e,g(e)\}$ as blue.
%Each edge of $C$ is either red or blue. 
Since $f, g$ are injections,
for each Vertex $e \in \LL$ belonging to $C$ we must have two Edges
adjacent to $e$, one being $\{e,f(e)\}$ and the other $\{e,g(e)\}$. 
Hence the degree of every $e \in \LL$ is $2$.
So $C$ has exactly $2j$ Edges.
On the other hand, each $e' \in \RR(f,g)$ has degree $1$ or $2$.
If there are $k_i$ Vertices in $\RR(f,g)$, $i=1,2$, then,
counting Edges again, $k_1+2k_2=2j$. 
Since $C$ is connected, it is easy to see that $k_1= 0$ or $2$.
Hence $k_2=j$ or $k_2 = j-1$.
Therefore, $k=k_1+k_2=1+j$ or $j$.
\end{proof}

%A component $C \in \CC_{1,1}$ consisting of a single Edge is referred to
%as a {\em basic component}.
%A component $C \in \CC_{j,j+1}$ for some $j \ge 1$ is referred to
%as a {\em zigzag component} because it is necessarily of the type
%depicted in Figure \ref{fig2} below.
%A component $C \in \CC_{j,j}$ for some $j \ge 2$ is 
%referred to as a {\em cycle component} because it is necessarily a cycle.
Lemma \ref{classlem} says that there are three kinds of connected components
of the auxiliary edge graph: the elements of $\CC_{1,1}$,
the elements of $\CC_{j,j+1}$, $j\ge 1$,
and the elements of $\CC_{j,j}$, $j \ge 2$. They look as
in Figure \ref{fig2}.

\begin{figure}
\begin{center}
\includegraphics[width=0.5\textwidth]{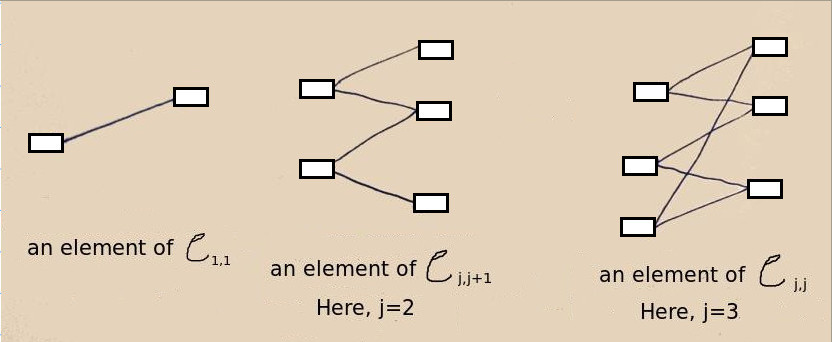}
\captionof{figure}{The only possible types of connected components of $\TT(f,g)$}
% are of three types:
%basic components (containing exactly one VERTEX from $\LL$ and one from $\RR(f,g)$),
%the zigsags (containing $j$ VERTICES from $\LL$ and $j+1$ from $\RR(f,g)$)
%and the cycles (containing $j\ge 2$ VERTICES from $\LL$ and $j$ from $\RR(f,g)$).
\label{fig2}
\end{center}
\end{figure}

\subsubsection{\bfseries Combinatorial estimates}
\label{CoCo}
The expression for $\E N^2$ involves a sum over pairs $(f,g)$ of
injections from $U$ to $V$. We will need to partition this set
as
\[
\II_{U,V} \times \II_{U,V} = \bigcup_{r=0}^m \HH_r,
\]
where 
\begin{equation}
\label{Hrdef}
\HH_r=\left\{(f,g) \in \II_{U,V} \times \II_{U,V}:\, |f(U)\cap g(U)|=r\right\},
\quad 0\le r \le m,
\end{equation}
and then further partition $\HH_r$ as
\[
\HH_r = \bigcup_{\ell=0}^r \HH_{r,\ell},
\]
where
\begin{equation}
\label{Hrldef}
\HH_{r, \ell} :=\big\{ (f, g) \in \HH_r : \#\{u\in U: f(u) = g(u)\} = \ell\big\},
\quad 0 \le \ell \le r.
\end{equation}
The reason is that we will later need to break the sum in \eqref{N2bis}
according to the first partition and then according to the second.

\begin{lemma}[cardinality of $\HH_r$]                                             
\label{cardHr}
\begin{equation}
\label{meno0}
|\HH_r| = (n)_m \binom{m}{r} (m)_r (n-m)_{m-r}.
\end{equation}
\end{lemma}
\begin{proof}
We can pick $f$ in $(n)_m$ ways. We notice that any $g$ such
that $(f, g) \in \HH_r$, for the particular $f$ we selected,
can be written uniquely as $g = g_1 \cup g_2$ where
$g_1 : U_1 \to  f(U)$ and $g_2: U_2 \to V \backslash f(U)$ and $U_1$, $U_2$
are a partition of $U$ into $r$, $m-r$ elements respectively.
We can pick $U_1$, $U_2$ in $\binom{m}{r}$ ways. Then we can
select $g_1$ in $\II_{U_1, f(U)}$ in $(m)_r$ ways and $g_2$ in
$\II_{U_2, V \backslash f(U)}$ in $(n-m)_{m-r}$ ways.
\end{proof}

\begin{lemma}[upper bound for the cardinality of $\bigcup_{k=\ell}^r \HH_{r, k}$]
\label{cardHrl}                                                                   
\begin{equation}
\label{meno}
\left|\bigcup_{k=\ell}^r \HH_{r, k}\right|
\le (n)_m \binom{m}{\ell,r-\ell,m-r} (m-\ell)_{r-\ell} (n - m)_{m - r}.
\end{equation}
\end{lemma}
\begin{proof}
We estimate the size of $\bigcup_{k=\ell}^r \HH_{r,k}$ as follows.
We first select $f$ in $\II_{U,V}$ in $(n)_m$ ways.
Then we select a partition $U$ into 3 sets
$U_1$, $U_2$, $U_3$ with $\ell$, $r-\ell$ and $m-r$ elements respectively.
This can be done in $\binom{m}{\ell, r - \ell, m - r}$ ways.
We then pick
$g_1 : U_1 \to f U$ such that $g(u) := f(u)$ for all $u$ in $U_1$.
This choice is unique.
We then pick an injection $g_2: U_2 \to f(U) \backslash f(U_1)$.
This can be done in $(m-\ell)_{r-\ell}$ ways. 
Finally, we pick an injection $g_3 : U_3 \to V \backslash f(U)$.
This can be done in $(n-m)_{m-r}$ ways.
We notice that $g = \bigcup_{j=1}^3 g_j$ is a function such that
$(f, g) \in \bigcup_{k = \ell}^r \HH_{r, k}$.
Actually any such pair $(f, g)$ is picked {\em at least} once using the above procedure. 
(Indeed, for any $(f,g)$ in $\bigcup_{k = \ell}^r \HH_{r, k}$ we can
find $U_1, U_2, U_3$ as above.)
Hence the size of $\bigcup_{k = \ell}^r \HH_{r, k}$ is at least the
size of the possible selections via the procedure above.
\iffalse
    \[
\overline h_{r,\ell} = 
        %\sum_{k =\ell}^r h_{r,k} =
        \left| \bigcup_{k = \ell}^r \HH_{r, k} \right| \le
        (n)_m \binom{m}{\ell, r - \ell, m - r} (m - \ell)_{r - \ell} (n - m)_{m - r}.
    \]
\fi
\end{proof}

\begin{remark}
\eqref{meno} holds with equality when $\ell=0$.
Indeed, the right-hand side of \eqref{meno} 
reduces to the right-hand side of \eqref{meno0}  when $\ell=0$.
\end{remark}

\subsubsection{\bfseries Estimates of sizes of connected component classes}
\label{SiSi}
In this section we compute or estimate several quantities related to
the sizes of the component classes $\CC_{j,k}(f,g)$ defined in \eqref{compcl}.
By Lemma \ref{classlem} we must have $k=j$ or $k=j+1$.
Recall the graph $\TT(f,g)$ from \eqref{Tgraph} with left Vertices $\LL$
and right vertices $\RR(f,g)$ as in \eqref{LeftV} and  \eqref{RightV}.
%Since $f, g$ are fixed elements of $\II_{U,V}$ 
%we omit them in various symbols and write, e.g., $\CC_{j,k}$ instead
%of $\CC_{j,k}(f,g)$.
\begin{lemma}
\label{escc}
Let $f, g \in  \II_{U,V}$. Set $\CC_{j,k}\equiv \CC_{j,k}(f,g)$, 
$\RR \equiv \RR(f,g)$. Then
\\
(i)
\begin{align}
|\CC| &= \sum_{j,k} |\CC_{j,k}| =  \sum_j (|\CC_{j,j}| + |\CC_{j,j+1}|),
\label{X1}
\\
|\LL| &= \sum_{j,k} j |\CC_{j,k}|
= \sum_j  (j |\CC_{j,j}| + j |\CC_{j,j+1}|) = \binom{m}{2}.
\label{X2}
\end{align}
(ii) If $|f(U)\cap g(U)|=r$ then
\begin{align}
 |\RR| = \sum_{j,k} k |\CC_{j,k}|
&= \sum_j  (j |\CC_{j,j}| + (j+1) |\CC_{j,j+1}|) 
= 2 \binom{m}{2} - \binom{r}{2},
\label{X3}
\\
\bigg| \bigcup_{j \ge 1} \CC_{j,j+1}\bigg| &= \binom{m}{2}-\binom{r}{2}.
\label{X4}
\end{align}
(iii) If, in addition to the assumption of (ii), 
the set $\{u \in U:\, f(u)=g(u)\}$ has size $\ell$, then
\begin{align}
\binom{\ell}{2} \le \left| \CC_{1,1} \right| 
&\le \binom{\ell}{2}+ \frac{1}{2}(r-\ell),
\label{X5}
\\
\bigg| \bigcup_{j \ge 2} \CC_{j,j}\bigg|  
& \le \frac12\left\{ \binom{r}{2} - \binom{\ell}{2} \right\}.
\label{X6}
\end{align}
\end{lemma}

\begin{proof}
Since the $\CC_{j,k}$ are pairwise disjoint with union $\CC$, and since,
by Lemma \ref{classlem}, $k$ can only be $j$ or $j+1$,
\eqref{X1} follows.
Since different components do not share any vertices, we have that
that $\sum_{j,k} j |\CC_{j,k}|$  is the total number $|\LL|$ of left Vertices
which, by the definition of $\LL$ in \eqref{LeftV},
is $\binom{m}{2}$, showing \eqref{X2}.
Similarly, $\sum_{j,k} k |\CC_{j,k}|$ is the total number $|\RR|$ 
of right Vertices. But $\RR$ is the union of $f(\LL)$ with $g(\LL)$,
hence \eqref{X3} follows by inclusion-exclusion, along with the 
fact that 
\[
f(\LL) \cap g(\LL) = \PP_2(f(U)\cap g(U)).
\]
%Hence the first equalities in \eqref{X2} and \eqref{X3} follow.
%But the set $\LL=\PP_2(U)$ of left Vertices has size $\binom{m}{2}$
%since $U$ has size $m$. Hence \eqref{X2} holds.
\eqref{X4} is obtained by subtracting \eqref{X2} from \eqref{X3}.

To prove \eqref{X5} we argue as follows.
$\CC_{1, 1}$ contains all sets $\{e, e'\}$ such that
$e' = f(e) = g(e)$. Let $L = \{u \in U : f(u) = g(u)\}$ and
$e = \{x, y\}$ where $x, y \in U$.
Then we have that $f(\{x, y\}) = g(\{x, y\})$ if and only if either
$f(x) = g(x)$ and $f(y) = g(y)$ or $f(x) =  g(y)$ and $f(y) = g(x)$.
Picking $e = \{x, y\} \in \PP_2(L)$ guarantees that $f(e) = g(e)$.
The condition $f(x) = g(y)$ and $f(y) = g(x)$ can be satisfied by
at most $\frac{1}{2} (r - \ell)$ ways. This is due to the fact that
both $x$, $y$ must not be in $L$ and $y = f^{-1}(g(x))$ and therefore
the number of equivalence classes in $f(U) \cap g(U) \setminus L$ with
$2$ elements is an upper bound for the ways we can pick $x, y$
satisfying $f(x) = g(y)$ and $f(y) = g(x)$.
\\
Finally, for \eqref{X6},
we argue as follows.
The set $\bigcup_{j \ge 2} \CC_{j,j}$ contains connected components
that are cycles in $\TT$.
We map this set injectively into a set of equivalence classes.
%Note that each right VERTEX $e$ of a cycle is of degree 2.
%Equivalently, $e=f(e_1)=g(e_2)$ for $e_1 \not = e_2$.
A right Vertex of any cycle is necessarily an element of the set
\[
\RR_2(f,g) 
= \{e \in f(\LL) \cap g(\LL): f^{-1}(e) \neq g^{-1}(e)\}.
\]
Call two elements of $\RR_2(f,g)$ equivalent if they are both
Vertices of the same connected component. Also, each cycle corresponds
to a component with at least two elements of $\RR_2(f,g)$. 
But there can be at most
$\frac12 |\RR_2(f, g)| = \frac12 \{\binom{r}{2} - \binom{\ell}{2}\}$ equivalence
classes with at least two elements in $\RR_2(f, g)$ when $(f, g)$ are in
$\HH_{r, \ell}$.
%Call two elements of $\RR_2(f,g)$ equivalent if they are both
%VERTICES of the same connected component. 
%The number of equivalence classes
%is equal to $\frac12 \{ \binom{r}{2}- \binom{\ell}{2}\}$.
%we count a certain set of equivalence
%classes and show that there is a  function from this set to the set of cycles.
%A cycle is a component of $\TT(f,g)$ that contains some $e \in \RR(f,g)$
\end{proof}

\subsubsection{\bfseries Assembling the pieces}
\label{SSSsubsec}
We return to the ratio $\E N^2/(\E N)^2$ and establish a 
non-asymptotic upper bound.
\begin{proposition}
\label{esoteric}
Let $X, Y$ be independent $G(U,p)$, $G(V,1/2)$ random graphs
where $U, V$ are sets of sizes $m,n$, respectively.
Define
\begin{equation}
\label{SSS}
\SSS := \frac{1}{(n)_m^2} \sum_{r=0}^m 
2^{\binom{r}{2}} \sum_{(f,g) \in \HH_r} (\pmax)^{\binom{m}{2}-|\CC(f,g)|},
\end{equation}
where $\widehat p = \max(p,1-p)$, where $\HH_r$ is given by \eqref{Hrdef},
and where $|\CC(f,g)|$ is the number of connected components of
the auxiliary edge graph $\EE(f,g)$.
Then
\[
\frac{\E N^2}{(\E N)^2} \le \SSS.
\]
\end{proposition}
%We have the following upper bound for 
%We proceed to find an upper bound 
%In view of \eqref{CS}, the theorem will be proved if we show that,
%with $m(n) = \lfloor m_*(n)-(C_n/\log n)\rfloor$, we have
%\[
%\limsup_{n \to \infty} \frac{ \E N^2}{(\E N)^2} \le 1.
%\]
\begin{proof}
%The denominator has been exactly computed in \eqref{ENone}.
%We thus need good upper bounds for the numerator $\E N^2$.
%Letting $\pmax = \max(p, 1-p)$, we upper-bound 
%the terms $\tau_{j,k}$ in the expression \eqref{N2bis} for $\E N^2$:
Since
\[
\tau_{j,k} = (p^j+(1-p)^j) 2^{-k} \le \pmax^{j-1} 2^{-k},
\]
expression \eqref{N2bis} for $\E N^2$ gives
\begin{align*}
\E N^2 &\le \sum_{f,g \in \II_{U,V}} \prod_{j,k}
\left(\pmax^{j-1} 2^{-k}\right)^{|\CC_{j,k}(f,g)|}
\\
&= \sum_{f,g \in \II_{U,V}} 
(\pmax)^{\sum_{j,k} (j-1) |\CC_{j,k}(f,g)|}
\times 2^{-\sum_{j,k} k |\CC_{j,k}(f,g)|}.
\end{align*}
The exponent of $\pmax$ equals
\[
\sum_{j,k} j |\CC_{j,k}(f,g)|  - \sum_{j,k} |\CC_{j,k}(f,g)| 
= \binom{m}{2} - |\CC(f,g)|,
\]
from \eqref{X2} and \eqref{X1}.
Identity \eqref{X3} tells us that the exponent of $2^{-1}$ equals
\[
\sum_{j,k} k |\CC_{j,k}(f,g)| = 2 \binom{m}{2} - \binom{r(f,g)}{2},
\]
where 
\[
r(f,g) = |f(U)\cap g(U)|.
\]
This leads to
\begin{equation*}
%\label{N2tris}
\E N^2  \le
2^{-2\binom{m}{2}}
\sum_{f,g \in \II_{U,V}}                                                          
(\pmax)^{\binom{m}{2}-|\CC(f,g)|} \times 2^{\binom{r(f,g)}{2}}.
\end{equation*}
Dividing this by the square of $\E N = (n)_m e^{-\binom{m}{2}}$ we have 
\[
\frac{\E N^2}{(\E N)^2} \le
\frac{1}{(n)_m^2} \sum_{f,g \in \II_{U,V}} 
(\pmax)^{\binom{m}{2}-|\CC(f,g)|} 2^{\binom{r(f,g)}{2}}.
\]
Since, by definition. $\HH_r$ is the set of pairs $(f,g)$
of injections from $U$ into $V$ such that $r(f,g)=r$,
we immediately have that the right-hand side of the last display is
$\SSS$.
\end{proof}

The proof of Theorem \ref{thm1}(I) will be complete
if we show that
\begin{equation}
\label{Sdone}
\limsup_{n \to \infty} \SSS \le 1.
\end{equation}
To achieve this, fix $0<c<1$ and write
\begin{equation}
\label{Ssplit}
\SSS = \SSSone + \SSStwo,
\qquad \SSSone =: \frac{1}{(n)_m^2} \sum_{0\le r \le cm}
2^{\binom{r}{2}} \sum_{(f,g) \in \HH_r} (\pmax)^{\binom{m}{2}-|\CC(f,g)|}.
\end{equation}

We now prove that, under the assumptions for phase I, 
the upper bound $\SSS$ is asymptotically below 1.

\begin{proposition}%[The significant term]
\label{S1part}
Let $c$ be fixed, $1/2<c<1$.
Define
\iffalse
\begin{equation}
\label{Ssplit}
\qquad 
\SSSone := \frac{1}{(n)_m^2} \sum_{0\le r \le cm}
2^{\binom{r}{2}} \sum_{(f,g) \in \HH_r} (\pmax)^{\binom{m}{2}-|\CC(f,g)|},
\qquad
\SSStwo := \SSS-\SSSone.
\end{equation}
\fi
%Let $1/2<c<1$.
With $m=m(n) = \left\lfloor 2 \log_2 n+1 - \frac{C_n}{\log n} \right\rfloor$, 
where $C_n \to \infty$ such that $C_n=o(\log n)$, we have
\begin{enumerate}[\rm (i)]
\item
$\limsup_{n \to \infty} \SSSone = 1.$
\item
$\limsup_{n \to \infty} \SSStwo = 0.$
\end{enumerate}
\end{proposition}

\begin{proof}[Proof of (i)]
The number $|\CC(f,g)|$ of connected components of $\TT(f,g)$
cannot exceed the cardinality $\binom{m}{2}$ of $\LL$. Therefore the exponent
of $\pmax$ in \eqref{Ssplit} is nonnegative. 
Since $\pmax < 1$, we immediately obtain
\[
\SSSone \le \sum_{r\le cm} \frac{2^{\binom{r}{2}}}{(n)_m^2} |\HH_r|.
\]
We computed $\HH_r$ in Lemma \ref{cardHr}, so
\[
\SSSone 
\le 1+ \sum_{1 \le r\le cm} 
2^{\binom{r}{2}}~ \frac{\binom{m}{r} (n-m)_{m-r} (m)_r}{(n)_m}
= 1+ \sum_{1 \le r\le cm} 
2^{\binom{r}{2}}~ \frac{(n-m)_{m-r} (m)_r^2}{(n)_m r!}.
\]
%$f$ [this can be done in $(n-r)_{m-r}$ ways]
%and then its remaining $r$ values from the range of $f$ this can be done in
%$(m)_r$ ways].
Using $(m)_r \le m^r$ and $r!\ge 1$ we get
\[
\SSSone 
\le 1 + \sum_{1\le r\le\lceil cm\rceil} 2^{\binom{r}{2}}
m^{2r} \frac{(n-m)_{m-r}}{(n)_m}.
\]
Now we calculate
\begin{align*}
\frac{(n-m)_{m-r}}{(n)_m}
&= \frac{(n-m) \cdots (n-2m+r+1)}{n(n-1) \cdots (n-m+1)} \\
&\le \frac{(n-m)^{m-r}}{(n-m)^m} = \left( \frac{1}{n-m} \right)^r \\
&= \left( \frac{1}{n} \right)^r \left( \frac{1}{1 - \frac{m}{n}} \right)^r \\
&\le \frac{1}{n^r} \left( \frac{1}{1 - \frac{m}{n}} \right)^m.
\end{align*}
But the term $\left( \frac{1}{1 - \frac{m}{n}} \right)^m$ is $1 + o(1)$ as
$n \rightarrow \infty$, because we have $m = O(\log n)$. Hence, we obtain
\[
\SSSone 
\le 1 + (1 + o(1)) \sum_{1\le r \le \lceil cm \rceil} 2^{\binom{r}{2}}
\frac{m^{2r}}{n^r}.
\]
We use convexity to show that the last sum tends to $0$. Letting
\[
a_r := 2^{\binom{r}{2}} \frac{m^{2r}}{n^r},
\]
we have that $a_{r+1}/a_r = (m^2/n) 2^r$ is increasing in $r$ and so
$a_r \le  a_1 \vee a_{\lceil cm \rceil}$ for all $1 \le r \le \lceil cm \rceil$,
which gives 
$\SSSone \le 1+e^{o(1)} \lceil cm \rceil (a_1\vee a_{\lceil cm \rceil})$.
It therefore suffices to show that both 
$m a_1 \to 0$ and $m a_{\lceil cm \rceil} \to 0$.
With our choice for $m=m(n)$, we immediately have
$m a_1 = {m^3}/{n} \to 0$.
%\[
    %m a_{cm} = 2^{\binom{cm}{2}} \frac{m^{2cm + 1}}{n^{cm}}
%\]
We next have
\begin{align*} 
\log_2 (m a_{\lceil cm\rceil }) 
&= \frac12 \lceil cm\rceil \big(\lceil cm\rceil-1\big)
+\big(2\lceil cm\rceil+ 1\big) \log_2 m -\lceil cm\rceil\log_2 n
\\ 
&\le \frac12 \lceil cm\rceil \big(\lceil cm\rceil-m\big)
+ \big(2\lceil cm\rceil+ 1\big) \log_2 m,
\end{align*}
where we used that, for $m=m(n)$ as in the theorem statement, $m< 2 \log_2 n+1$.
Since $c<1$, we have $\lceil cm\rceil-m \to \infty$
and since the last term in the above display converges to $\infty$ much
faster than $m^2$ we immediately obtain $\log_2 (m a_{\lceil cm\rceil }) 
\to -\infty$, as needed.
\end{proof}

\iffalse
For the second term, $\SSStwo$, in \eqref{Ssplit}, we have:
\begin{lemma}[The negligible term]
\label{S2part}
Suppose that $1/2 < c < 1$.
With $m=m(n) = \left\lfloor 2 \log_2 n+1 - \frac{C_n}{\log_2 n} \right\rfloor$, 
where $C_n \to \infty$, $C_n=o(\log n)$, we have
\[
\SSStwo:=\frac{\pmax^{\binom{m}{2}}}{(n)_m^2}
\sum_{r > cm} 2^{\binom{r}{2}} \sum_{(f,g) \in \HH_r} (\pmax)^{-|\CC(f,g)|}
\to 0,
\]
as $n \to \infty$.
\end{lemma}
\fi

\iffalse
\[
-\frac{
\Gamma \! \left(n +1\right) 
\Gamma \! \left(m -l \right) 
\left(\mathit{multinomial}\! \left(m , l +1, r -l -1, m -r \right)
+\mathit{multinomial}\! \left(m , l , r -l , m -r \right)
\left(-m +l \right)\right)}
{
\Gamma \! \left(n -2 m +r +1\right) 
\Gamma \! \left(m -r +1\right)
}
\]
{\color{Gray!90}
\footnotesize
In fact, we have
\[
-\frac{
n!
(m-\ell)!
}
{
(n -2 m +r)!
(m -r)!
}
\left\{
(m-\ell)
\binom{m}{\ell,r-\ell,m-r}
- \binom{m}{\ell+1,r-\ell-1,m-r}
\right\}
\]
}
\fi

\begin{proof}[Proof of (ii)]
Recall the definition \eqref{Hrldef} of $\HH_{r,\ell}$
and write $\SSStwo$ as
\[
\SSStwo = \frac{1}{(n)_m^2}
\sum_{r > cm} 2^{\binom{r}{2}} 
\sum_{\ell=0}^r \sum_{(f,g) \in \HH_{r,\ell}} \pmax^{\binom{m}{2}-|\CC(f,g)|}.
\]
An upper bound for $|\CC(f,g)$, when $(f,g) \in \HH_{r,\ell}$, is
obtained by using
\eqref{X1}, \eqref{X5}, \eqref{X6} and \eqref{X4}
of Lemma \ref{escc}:
\begin{align*}
|\CC(f,g)| &= |\CC_{1,1}(f,g)| + \sum_{j \ge 2} |\CC_{j,j}(f,g)|
+ \sum_{j \ge 1} |\CC_{j,j+1}(f,g)|
\\
&\le \left[\binom{\ell}{2} + \frac12(r-\ell)\right]
+ \left[\frac12 \binom{r}{2} -\frac12\binom{\ell}{2}\right]
+ \left[\binom{m}{2} - \binom{4}{2}\right]
\\
&= \binom{m}{2} + \frac12\binom{\ell}{2} - \frac12 \binom{r}{2} + \frac12(r-\ell)
 \le \binom{m}{2} - \frac14 (r-\ell)(r-3).
\end{align*}
We thus have
\[ 
\SSStwo \le \frac{1}{(n)_m^2}
\sum_{r > cm} 2^{\binom{r}{2}} \sum_{\ell=0}^r 
(\pmax)^{\frac14 (r-\ell) (r-3)} |\HH_{r,\ell}|.
\]
Writing $|\HH_{r,\ell}| \le \left|\bigcup_{k=\ell}^r \HH_{r, k}\right|$ 
and upper-bounding this as in \eqref{meno}, we obtain
%by the estimate obtained
%in \eqref{meno} we further obtain
\[
\SSStwo \le 
\sum_{r > cm} 
\frac{(n - m)_{m - r}}{(n)_m}
2^{\binom{r}{2}} \binom{m}{r}
\sum_{\ell=0}^r \binom{r}{\ell} (m-\ell)_{r-\ell}
\, (\pmax)^{\frac14 (r-\ell) (r-3)} .
\]
Since $r>cm$, the last term is upper-bounded by
$(\pmax)^{\frac14 (r-\ell) (cm-3)}$; using also
$(m-\ell)_{r-\ell} \le m^{r-\ell}$, we upper-bound the last sum
by
\[
\sum_{\ell=0}^r \binom{r}{\ell} \left(m \pmax^{(cm-3)/4}\right)^{r-\ell}
= \left(1+m \pmax^{(cm-3)/4}\right)^m
\to 1, \quad \text{ as } n \to \infty.
\]
We thus obtain
\[
\SSStwo \le (1+o(1)) \sum_{r > cm}
\frac{(n - m)_{m - r}}{(n)_m} 
\binom{m}{r}
2^{\binom{r}{2}}
\le (1+o(1)) \sum_{r > cm} \frac{1}{n^r} \binom{m}{r}
2^{\binom{r}{2}},
\]
where we used the facts that $(n - m)_{m - r} \le n^{m-r}$
and that $\lim_{n \to \infty} (n)_m/n^m =1$.
We use a convexity argument again.
Consider
\[
b_r = b_r(n) = \frac{1}{n^r} \binom{m}{r} 2^{\binom{r}{2}},
\quad \lfloor cm \rfloor + 1 \le r \le m,
\]
as a function of $r$.
(Note the condition $r > cm \ge \lfloor cm \rfloor$ is equivalent to
$r \ge \lfloor cm \rfloor + 1$ because $r$ is integer.)
We have
\[
\frac{b_{r+1}}{b_r} = \frac1n\, \frac{m-r}{r+1}\, 2^r.
\]
Setting
$\psi (x) := \frac{x}{x+1}$,
an increasing function on $x > -1$,
we obtain %, for all $\lfloor cm \rfloor + 1 \le r \le m-2$,
\begin{align*}
%L_r:=
\frac{b_{r+2}}{b_{r+1}} \bigg/ \frac{b_{r+1}}{b_r} 
= 2\, \frac{r+1}{r+2}\, \frac{m-r-1}{m-r} 
&= 2\, \psi(r+1)\, \psi(m - r - 1)
\\
& \ge 2 \psi(\lfloor cm \rfloor + 2) \psi(2) =
\frac43 \psi(\lfloor cm \rfloor + 2) > 1,
\end{align*}
for $\lfloor cm \rfloor + 1 \le r \le m - 3$ and $m$ large enough,
because $\psi(x) \to 1$ as $x \to \infty$. We deduce that
the sequence of ratios $\frac{b_{r+1}}{b_r}$ is increasing for
$\lfloor cm \rfloor + 1 \le r \le m - 1$. We notice that
\[
\frac{b_{\lfloor cm \rfloor + 2}}{b_{\lfloor cm \rfloor + 1}} =
\frac1n \frac{m - \lfloor cm \rfloor - 1}{\lfloor cm \rfloor + 1}
2^{\lfloor cm \rfloor + 1}.
\]
The term $\frac{m - \lfloor cm \rfloor - 1}{\lfloor cm \rfloor + 1}$
converges to $\frac{1 - c}{c}$. 
Moreover, $m = \lfloor 2 \log_2 n + 1
- \frac{C_n}{\log n} \rfloor \ge 2 \log_2 n$ for $n$ large enough, using the
fact that $\frac{C_n}{\log n } \to 0$ as $n \to \infty$. So, it follows that
\[
\frac{b_{\lfloor cm \rfloor + 2}}{b_{\lfloor cm \rfloor + 1}} \ge
\left(\frac{1-c}{c} + o(1)\right) \frac{1}{n} 2^{2c \log_2 n} =
\left(\frac{1-c}{c} + o(1)\right) n^{2c - 1}
\]
but the latter term converges to $\infty$ because 
$c > \frac12$. So we can pick
some $a > 1$ such that for $n$ large enough
\[
\frac{b_{\lfloor cm \rfloor + 2}}{b_{\lfloor cm \rfloor + 1}} \ge a > 1.
\]
Now using the fact that the sequence of ratios is increasing up to $m-2$
we get
\[
\frac{b_{r+1}}{b_r} \ge a,
\]
for all $\lfloor cm \rfloor + 1 \le r \le m-2$, which implies that
\[
b_r \le \left(\frac{1}{a}\right)^{m - 1 - r} b_{m - 1}.
\]
Since $m > 2 \log_2 n$, we have
\[
\frac{b_m}{b_{m-1}} = \frac{1}{n} \frac{1}{m} 2^m
> \frac{n}{m}.
\]
%\[
%\frac{b_m}{b_{m-1}} > \frac{1}{n m} 2^{2 \log_2 n} = \frac{n}{m}
%\]
%where the fraction $\frac{n}{m} \to \infty$ as $n \to \infty$.
Therefore
\[
b_r \le \left(\frac{1}{a}\right)^{m - 1 - r} b_m, \quad
\lfloor cm \rfloor + 1 \le r \le m - 1.
\]
This implies that
\[
\sum_{r > cm} b_r = b_m + \sum_{r = \lfloor cm \rfloor + 1}^{m-1} b_r
\le b_m + b_m \sum_{r = \lfloor cm \rfloor + 1}^{m-1} \left(\frac{1}{a}\right)^{m - 1 - r}
\le b_m \left(1 + \frac{1}{1 - (1/a)}\right).
\]
So to show $\SSStwo \to 0$ it suffices to show that $b_m \to 0$. 
To this end we compute
\[
2 \log_2 b_m = m ( m - 1 ) - 2 m \log_2 n = m (m - 2 \log_2 n-1).
\]
However,
\[
m = \lfloor 2 \log_2 n+1 - \frac{C_n}{\log n} \rfloor \le
2 \log_2 n+1 - \frac{C_n}{\log n}
\]
and therefore
\[
2 \log_2 b_m \le - C_n \frac{m}{\log n}.
\]
Since $\frac{m}{\log_2 n} \to 2$ and $C_n \to \infty$, we conclude that
$2 \log_2 b_m \to - \infty$ and therefore $b_m \to 0$ which completes the proof.
\end{proof}

Proposition \ref{S1part} immediately implies that \eqref{Sdone} holds,
and so the proof of Theorem \ref{thm1} is complete.

\section{\bfseries\scshape The common subgraph problem}
Assuming now that the two vertex sets $U, V$, have the same size $n$, our goal
is to discover the size of the largest common subgraph as $n \to \infty$.
We shall again prove a phase transition phenomenon occurs, as stated in
Theorem \ref{thm2}.
We let $X=X_n$, $Y=Y_n$ be independent random graphs with 
laws $G(U,p)$, $G(V,q)$, respectively. 
We now allow $p$ and $q$ to be different probabilities, both strictly
between $0$ and $1$, but restricted in the region defined by \eqref{thespis} 
and depicted in Figure \ref{figpq}. The method is analogous to that
of the embedding problem, but the details of the analysis are more complicated.
%We assume that $U, V$ are disjoint.

We start with some terminology and notation.
We say that $f$ is {\em partial function} from $U$ to $V$
if there are sets $U' \subset U$ and $V' \subset V$ such that
$f$ is a function from $U'$ {\em onto} $V'$. We denote $U'$ by $\dom f$
and $V'$ by $\ran f$.
We are interested in partial functions that are also injections.
Recalling that $\II_{U',V}$ is the collection of all injections 
from $U'$ to $V$, the set
\[
\JJ_{U,V,m} := \bigcup_{U' \in \PP_m(U)} \II_{U',V}
\]
is the collection of all partial functions 
from $U$ to $V$ that are injections and have domain of size $m$.
%So $f: \dom f \to \ran f$ is a bijection.
Since there are $n \choose m$ ways to choose a subset of $U$ of size $m$
and then there are $(n)_m$ injections from the chosen set into $V$, it
follows that
\begin{equation}
\label{JJsize}
|\JJ_{U,V,m}| = \binom{n}{m} \, (n)_m = \frac{(n)_m^2}{m!}.
\end{equation}
Recall the definition of $m$-isomorphism; see Def.\ \ref{misodef}.
%Whether $f$ is an $m$-isomorphism or not is expressed via the indicator
Defining
\begin{equation}
\label{Jf}
J_f := \I_{f (X_n^{\dom f}) = Y_n^{\ran f}}
\end{equation}
(recall that the induced subgraph on a set $A$ of a graph $\Gamma$ is denoted by $\Gamma^A$),
we see that
\[
X_n \jso{m} Y_n \iff \exists f \in \JJ_{U,V,m} ~~ J_f=1.
\]
Thus
\begin{equation} \label{NN1}
N = \sum_{f \in \JJ_{U,V,m}} 
J_f
\end{equation}
is the number of $m$-isomorphisms between the two random graphs,
and
\[
X_n \jso{m} Y_n \iff N>0.
\]
With $\tau=pq+(1-p)(1-q)$ we have, for any $f \in \JJ_{U,V,m}$,
\begin{align*}
\E J_f 
&= \P(f (X_n^{\dom f}) = Y_n^{\ran f}) 
\\
& = \P\big(X_n(e) = Y_n(f(e)) \text{ for all } e \in \PP_2(\dom f)\big)
= \tau^{\binom{m}{2}},
\end{align*}
so
\begin{equation}
\label{crito}
\E N = |\JJ_{U,V,m}|\, \E J_f
=  \binom{n}{m} \, (n)_m\, \tau^{\binom{m}{2}}.
\end{equation}
As usual, we shall sometimes be writing $X$ instead of $X_n$, etc.

\subsection{\bfseries Phase II of the common subgraph problem}
Phase II, by definition, refers to the asymptotic regime
where the probability of existence of an $m$-isomorphism
tends to $0$ as $n \to \infty$ for some sequence $m=m(n)$.
A sufficient condition is easily obtained below.

Recall, from \eqref{Wdef} and \eqref{mstar} that
$W(x)= x+ 2\lambda \log x + \frac{\lambda}{x} \log(2\pi x)$, with $\lambda
= 1/\log(1/\tau)$, and
$W(m_*(n)) = 4 \lambda \log n + 2 \lambda+1$.
Observe that $W$ is strictly increasing (and concave),
with $W(1) \ge 1$, $W'(x) \ge 1$ for all $x$ and $W(x) \to \infty$
as $x \to \infty$.

\begin{lemma}[condition for phase II of the graph isomorphism problem]
\label{ENasymptlem}
%Let $0<p,q<1$, $\tau=pq+(1-p)(1-q)$, $\mylog x := \log_{1/\tau} x$.
If $m(n)= \lceil m_*(n)+(C_n/\log n)\rceil$ where $C_n \to \infty$ 
and $C_n/\log n \to 0$,
then $\E N \to 0$.
\end{lemma}
\begin{proof}
With $m=m(n)$ as in the statement, we have
\begin{equation}
\label{ENasympt}
\E N  = \frac{(n)_m \cdot (n)_m}{m!} \tau^{\frac12 m(m-1)}
\sim \frac{n^{2m}}{(m/e)^m \sqrt{2\pi m}} \tau^{\frac12 m(m-1)} =:b(n),
\end{equation}
where we took into account that $m(n) = \lceil m_*(n)+(C_n/\log n)\rceil
= O(\log n)$ (because $C_n/\log n \to 0$ and because of Remark \ref{Wrem}) and
used Stirling's approximation and that $(n)_m \sim n^m$, as $n \to \infty$.
Recalling the definition of $W$ and $m_*$ from \eqref{Wdef} and \eqref{mstar}
and doing a little algebra we obtain
\begin{equation}
\label{weealg}
- \log b(n) = \frac{1}{2\lambda} \, m\, ( W(m) - W(m_*) ),
\end{equation}
a positive quantity because $\lambda=1/\log(1/\tau)>0$, $m> m_*$,
$W(m) > W(m_*)$.
Since, for $x>0$, 
\begin{equation}
\label{Wprime}
W'(x) = 1+\frac{\lambda}{x}\left(2+\frac1x - \frac{\log(2\pi x)}{x}\right)
\ge 1+ \frac{\lambda}{x} (2-2\pi e^{-2}) \ge 1
\end{equation}
(indeed, the bracketed expression in the second term 
achieves minimum at the point $x=e^2/2\pi$ and
equals $2-2\pi{e^{-2}} > 0$ at this point),
we have 
\[
W(m)-W(m_*) \ge m-m_*
\]
and so
\[
-\log b(n) \ge  \frac{1}{2\lambda} \, m_*\, (m -m_*)
\ge \frac{1}{2\lambda} \, m_*\, \frac{C_n}{\log n}.
%\to -\infty \text{ as } n \to \infty.
\]
Since $\lim_{x \to \infty} W(x)/x=1$, we have 
\begin{equation}
\label{4lambda}
\frac{m_*}{\log n} 
= \frac{W(m_*)}{\log n} \left(\frac{W(m_*)}{m_*}\right)^{-1}
= \frac{4\lambda \log n+2\lambda+1}{\log n} \left(\frac{W(m_*)}{m_*}\right)^{-1}
\to 4\lambda,
\end{equation}
%(Remark \ref{Wrem}) 
and so $-\log b(n) \to \infty$ which implies that $\E N \to 0$.
\end{proof}

\subsection{\bfseries Phase I of the common subgraph problem}
%Just as in the graph embedding problem, we will use the Cauchy-Schwarz inequality
%\eqref{CS} in order to show that the asymptotic lower bound on $\P(N>0)$
%is $1$, under the conditions for phase I.
%This requires a very precise estimate of $\E N^2$ which is
%greatly facilitated by means of the graph $\TT(f,g)$ introduced next.
We shall use the auxiliary edge graph device
in order to estimate $\E N^2$.
Recall that an edge graph is a graph whose Vertices are 
edges of $U$ or $V$, i.e.\ elements of $\PP_2(U) \cup \PP_2(V)$.

\subsubsection{\bfseries The auxiliary edge graph}
Fix $f,g \in \II_{U,V,m}$ and define the auxiliary edge graph
$\TT(f,g)$ as follows.
\begin{align*}
\LL(f,g) &= \PP_2(\dom f) \cup \PP_2(\dom g) \text{ (left Vertices)},
\\
\RR(f,g) &= \PP_2(\ran f) \cup \PP_2(\ran g) \text{ (right Vertices)}.
\end{align*}
The Vertex set of $\TT(f,g)$ is $\LL(f,g) \cup \RR(f,g)$ 
while its Edge set is
\[
\EE(f,g) = \big\{ \{e, f(e)\}:\, e \in \PP_2(\dom f) \big\}
\cup \big\{ \{e, g(e)\}:\, e \in \PP_2(\dom g) \big\}.
\]
Clearly, $\TT(f,g)$ is bipartite. 
The graph $\TT(f,g)$ is more involved than the one introduced in 
Section \ref{PH1emb} because $f, g$ are partial functions
and thus may not be defined on every element of $U$.
Note that we write $f(u)$ for the action of $f$ on a $u \in U$,
but we also write $f(e)$ for the action of $f$ on an edge $e$ of $U$.
Since $f : \dom f \to \ran f$ is a bijection, the map
$f: \PP_2(\dom f) \to \PP_2(\ran f)$ is also a bijection,
and hence $f^{-1}(e)$ is well-defined when $e \in \ran f$.
So if $\{e,e'\}$ is an Edge of $\TT(f,g)$, 
then $e'=f(e)$ (equivalently, $e=f^{-1}(e')$) 
or $e'=g(e)$  (equivalently, $e=g^{-1}(e')$).
We let
\[
\Deg_{f,g}(e) := \text{ Degree of Vertex $e$ in $\TT(f,g)$},
\]
that is, the number of Vertices $e'$ such that $\{e,e'\}$ is an Edge.
Clearly,
\[
\Deg_{f,g}(e) = 1 \text{ or } 2, 
\]
according as $f(e)=g(e)$ or $f(e) \neq g(e)$, respectively.

Let $\CC(f,g)$ be the collection of connected components of $\TT(f,g)$
and further let $\CC_{j,k}(f,g)$ be the collection of connected components
with $j$ left and $k$ right Vertices.

\begin{lemma}
\label{ccut}
(i) The set $\CC_{j,k}(f,g)$ is empty unless $j=k-1$ or $j=k$ or $j=k+1$.
\\
(ii) If $C \in \CC_{j,k}(f,g)$ with $|k-j|=1$ then $C$ is a path.
\\
(iii) If $C \in \CC_{j,j}(f,g)$ then $C$ is a path or a cycle.
\end{lemma}
\begin{proof}
%Denote by $\CC_{j, k} (f, g)$ the set of connected components
%that contain $j$ left VERTICES and $k$ right VERTICES, 
%Clearly, $\CC_{j, k} (f, g)$, $j \ge 1, k \ge 1$, form a partition of 
%$\CC(f,g)$.
%
Let $C \in \CC_{j, k}(f, g)$. For $d \in \{1,2\}$ denote by $j_d$, respectively
$k_d$ the number of left, respectively right, Vertices of $C$ of Degree $d$. So
\begin{equation}
\label{argu1}
j_1 + j_2 = j, \quad k_1 + k_2 = k.
\end{equation}
To each left Vertex of Degree $d$ of the bipartite graph $C$ there correspond $d$
Edges. Hence $j_1+2j_2$ is the number of Edges of $C$. It also equals $k_1+2k_2$. So
\begin{equation}
\label{argu2}
j_1 + 2 j_2 = k_1 + 2 k_2.
\end{equation}
Notice that $j_d+k_d$ is the number of Vertices of $C$ of Degree $d$.
Then \eqref{argu2} immediately gives 
\begin{equation}
\label{argu3}
j_1 + k_1 = 2 (k_2 - j_2 + k_1),
\end{equation}
that is, the number of Vertices of Degree $1$ is even.
\\
{\em Case 1:} $j_1+k_1=0$. From \eqref{argu3} we have $k_2=j_2$ and hence,
from \eqref{argu1}, $j=k$.
Thus, $C$ is a connected component with equal left and right Vertices and all
of Degree $2$. Hence $C$ is a cycle.
\\
{\em Case 2:} $j_1+k_1 > 0$. From \eqref{argu3}, $j_1+k_1 \ge 2$.
Let then $e, e'$ be distinct Vertices of $C$ Degree $1$ each.
Since $C$ is connected, there is a path $\Pi$ from $e$ to $e'$.
Every other Vertex on this path must have Degree 2. 
Since no Vertex in $\TT(f,g)$ has Degree larger than $2$, it follows that
$C=\Pi$. So $C$ itself is a path with $j_1+k_1=2$.
Since, from \eqref{argu1} and \eqref{argu2} we have
$j-k = \frac12(j_1-k_1),$
the only possible values of $j-k$ are $\pm1$ or $0$.
\end{proof}

Define, for all $j \ge 1$,
\begin{align*}
\CC_{j, j}^o (f, g) &= \{ C \in \CC_{j, j}(f, g) :\, 
\Deg_{f,g}(e) = 2 \text{ for all Vertices $e$ of $C$}\}
\\
\CC_{j, j}^* (f, g) &= \CC_{j, j}(f, g) \setminus \CC_{j, j}^o (f, g).
\end{align*}
By Lemma \ref{ccut}(iii), elements of $\CC^o_{j,j}(f,g)$ are cycles
and elements of $\CC_{j, j}^* (f, g)$ are paths. 
By Lemma \ref{ccut}(i), we write $\CC(f,g)$ as the union of
four classes,
\[
\CC(f,g) = \bigcup_{j \ge 1} \CC_{j,j+1}(f,g)
\cup \bigcup_{j \ge 1} \CC_{j+1,j}(f,g)
\cup \bigcup_{j \ge 1} \CC^*_{j,j}(f,g)
\cup \bigcup_{j \ge 1} \CC^o_{j,j}(f,g),
\]
this being a union of pairwise disjoint sets.
Moreover, by Lemma \ref{ccut}(ii), the elements of the
first three classes are paths and the elements of $\bigcup_j \CC^o_{j,j}(f,g)$ 
are cycles.
An illustration of this is in Figure \ref{fig22}.

\begin{figure}
\begin{center}
\includegraphics[width=0.9\textwidth]{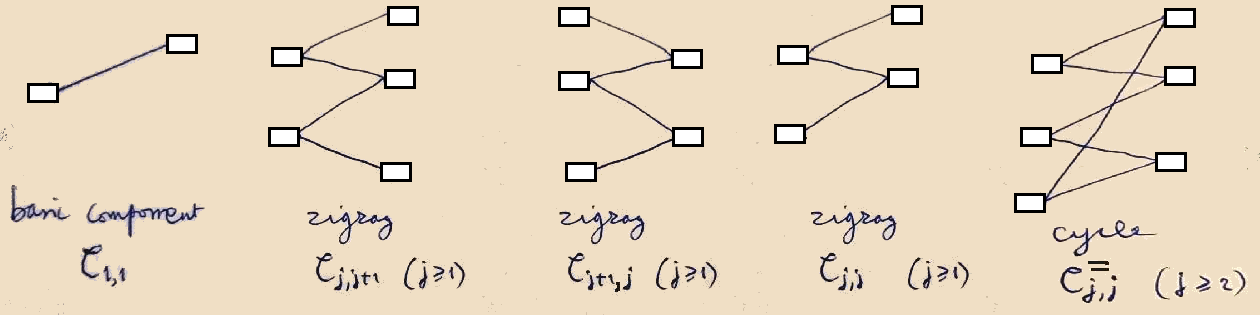}
\captionof{figure}{The connected components of $\TT(f,g)$}
\label{fig22}
\end{center}
\end{figure}

Recall that
$\tau_{j,k} \equiv \tau_{j,k}(p,q) = p^jq^k + (1-p)^j (1-q)^k$, as
in \eqref{taujk}. 
\begin{proposition}
\label{propN2tris}
Let $f, g$ be elements of $\JJ_{U,V,m}$.
If $J_f$, respectively, $J_g$,
is the indicator of the event that $f$, respectively $g$, is an $m$-isomorphism,
then
\begin{equation}
\label{N2tris} 
\E J_f J_g = \prod_{j,k \ge 1} \tau_{j,k}^{|\CC_{j,k}(f,g)|}.
\end{equation}
\end{proposition}
\begin{proof}
Let $G=(\VV_G, \EE_G)$ be a bipartite edge graph 
between $\PP_2(U)$ and $\PP_2(V)$. So
\begin{align*}
\VV_G &\subset \PP_2(U) \cup \PP_2(V),\\
\EE_G &\subset (\PP_2(U) \cap \VV_G) \times (\PP_2(V) \cap \VV_G).
\end{align*}
Define
\[
J(G) := \prod_{\{e, e'\} \in \mathcal{E}_G} \I_{X(e) = Y(e')}.
\]
%An example of such a graph is $\TT(f,g)$.
If $G, H$ are two bipartite edge graphs, let 
\[
G \cup H := (\VV_G \cup \VV_H, \EE_G \cup \EE_H),
\]
and observe that
\[
J(G)J(H) = J(G \cup H).
\]
For $f \in \JJ_{U,V,m}$, if $G_f$ is the edge graph with Vertex set 
$\PP_2(\dom f) \cup \PP_2(\ran f)$
and Edge set $\big\{\{e, f(e)\}: e \in \PP_2(\dom f)\big\}$, then
\[
J_f = J(G_f). 
\]
Observing that
\[
G_f \cup G_g = \TT(f,g),
\]
we obtain
\[
J_f J_g = J(\TT(f,g)).
\]
Since $\EE(f,g)$ is the Edge set of $\TT(f,g)$,
and since we can partition $\EE(f,g)$ into sets of edges belonging to
the connected components,
we further have
\[
J(\TT(f,g)) = \prod_{\{e,e'\} \in \EE(f,g)} \1_{X(e)=Y(e')}
= \prod_{C \in \CC(f,g)} \prod_{\{e,e'\} \in \EE_C} \1_{X(e)=Y(e')}
= \prod_{C \in \CC(f,g)} J(C) 
\]
Since the random variables $J(C)$, $C \in \CC(f,g)$, are independent,
we have
\begin{multline*}
\E J(\TT(f,g)) = \prod_{C \in \CC(f,g)} \E J(C)
\\
= \prod_{j,k \ge 1} \prod_{C \in \CC_{j,k}(f,g)} \P(X(e)=X(e') \text{ for all }
\{e,e'\} \in \EE_C)
= \prod_{j,k \ge 1} \tau_{j,k}^{|\CC_{j,k}(f,g)|}.
\end{multline*}
\end{proof}

\subsubsection{\bfseries Estimates of sizes of connected component classes}
We fix $f,g \in \JJ_{U,V,m}$ throughout this section. We define
\begin{align}
Z(f,g) &:= \left\{ u \in \dom f \cap \dom g:\, f(u) = g(u) \right\},
\label{Zfg}
\\
\ZZ(f,g) &:= \left\{ e \in \PP_2(\dom f) \cap \PP_2(\dom g) :\, 
f(e) = g(e) \right\},
\nonumber
\\
\LL_1(f,g) &:= \{e \in \LL(f,g):\, \Deg_{f,g}(e)=1\},
\nonumber
\\
\RR_1(f,g) &:= \{e \in \RR(f,g):\, \Deg_{f,g}(e)=1\}.
\nonumber
\end{align}  
Since $f, g$ won't change in this section,
we write $|\CC|$,  $|\CC_{j,k}|$, $|\LL|$, etc.,
instead of $|\CC(f,g)|$, $|\CC_{j,k}(f,g)|$, $|\LL(f,g)|$, etc. 

\begin{lemma}
\label{sigmaident}
\begin{align}
|\LL| &= 
\sum_{j\ge 1} (j |\CC_{j,j}| + j |\CC_{j,j+1}| + (j+1) |\CC_{j+1,j}| )
= 2 \binom{m}{2} - \binom{|Z|}{2},
\label{left}
\\
|\RR|&=\sum_{j\ge 1} (j |\CC_{j,j}| + j |\CC_{j+1,j}| + (j+1) |\CC_{j,j+1}| )
= 2 \binom{m}{2} - \binom{|\ZZ|}{2}.
\label{right}
\\
|\LL_1|&=\sum_{j\ge 1} (|\CC_{j,j}^*| + 2 |\CC_{j+1,j}|)
= |\ZZ| +  2 \binom{m}{2} - 2 \binom{|Z|}{2},
%\qquad \text{ no.\ of left VERTICES of DEG 1}
\label{left1}
\\
|\RR_1|&=\sum_{j\ge 1} (|\CC_{j,j}^*| + 2 |\CC_{j,j+1}|)
= |\ZZ| +  2 \binom{m}{2} - 2 \binom{\ZZ}{2}.
%\qquad \text{ no.\ of right VERTICES of DEG 1}
\label{right1}
\end{align}
\end{lemma}

\begin{proof}
(i)
The first equality in \eqref{left} is due to Lemma \ref{ccut}(i). 
The second equality follows by inclusion-exclusion.
Similarly for \eqref{right}.
\\
(ii)  
Every left Vertex of Degree 1 must belong to a component from the set $\CC_{j+1,j}$
or from the set $\CC_{j,j}^*$. Each $C \in \CC^*_{j,j}$ has exactly one left
Vertex of Degree 1; each $C \in \CC_{j+1,j}$ has exactly two left
Vertices  of Degree 1. See Figure \ref{fig22}. This proves the first equality 
in \eqref{left1}.
For the second equality note that
\[
\LL_1= \{e \in \PP_2(\dom f) \cup \PP_2(\dom g): \Deg(e)=1\}
= \ZZ
\cup (\PP_2(\dom f) \xor \PP_2(\dom g)),
\]
because Vertices in $\PP_2(\dom f) \xor \PP_2(\dom g)$ have degree 1.
Since the sets $\ZZ$ and $\PP_2(\dom f) \xor \PP_2(\dom g)$ are disjoint,
the second equality in \eqref{left1} follows.
Similarly for \eqref{right1}.
\end{proof}

\begin{corollary}
\begin{align}
&\sum_{j \ge 1} \left\{ |\CC_{j,j+1}| - |\CC_{j+1,j}|\right\} 
= \binom{d}{2}-\binom{r}{2}
\label{por1}
\\
&\sum_{j \ge 1} \left\{(j-1)|\CC_{j,j}^*| + j |\CC_{j,j}^o| 
+ (j-1)|\CC_{j,j+1}| + j |\CC_{j+1,j}| \right\} 
= \binom{r}{2} - |\ZZ|
\label{por2}
\\
&\sum_{j \ge 2} |\CC_{j, j}^o| \le \frac{1}{2} \left( \binom{r}{2} - |\ZZ| \right)
\label{por3}
\\
&
%\sum_{j \ge 1} ((j-1) (|\CC_{j, j}| + |\CC_{j, j+1}|) + j |\CC_{j+1, j}|)
%\ge \frac{1}{2} \left( \binom{r}{2} - |\ZZ| \right)
\sum_{j \ge 1} \left\{
(j-1) |\CC_{j, j}| + (j-1)|\CC_{j, j+1}| + j |\CC_{j+1, j}| \right\}
\ge \frac{1}{2} \left( \binom{r}{2} - |\ZZ| \right)
\label{por4}
\end{align}
\end{corollary}

\proof
Subtracting \eqref{right} from \eqref{left} we obtain \eqref{por1}.
Subtracting \eqref{right1} from \eqref{right} we obtain \eqref{por2}.
Since $\CC_{1, 1}^o = \varnothing$, \eqref{por3} follows from \eqref{por2}.
Using \eqref{por2} and \eqref{por3} we obtain \eqref{por4}.

\begin{lemma} 
\label{lambduplem}
\begin{equation}
\label{lambdupeq}
\binom{|Z|}{2}  \le |\ZZ| \le \binom{|Z|}{2} + \frac{1}{2} (r - |Z|).
\end{equation}
\end{lemma}
\begin{proof}
The first inequality is due to the inclusion $\PP_2(Z) \subset \ZZ$.
For $e = \{ x, y\} \in \ZZ$ we have 
\[
f(\{ x, y \}) = g(\{ x, y \}). 
\]
So, it follows that either
\[
f(x) = g(x) \, , \, f(y) = g(y)
\]
or
\[
f(x) = g(y) \, , \, f(y) = g(x).
\]
The first condition implies that $e \in \PP_2(Z)$ while the second condition
implies that $f(y) = f \left( g^{-1} \left( f(x) \right) \right)$ and $x, y \notin Z$.
The first condition is satisfied by all the elements of $\PP_2(Z)$.
The second condition can be satisfied by at most 
$\frac{1}{2} (r-|Z|)$ elements since $f(x), f(y)$ is a
unique pair of two distinct elements in $\ran f \cap \ran g \setminus Z'$ where
$Z' = f(Z) = g(Z)$. Therefore it follows that
\[
\left| \ZZ \right| \le \binom{ |Z| }{2} + \frac{1}{2} (r - |Z|).
\]
\end{proof}

\subsubsection{\bfseries Combinatorial estimates}
\label{seccombest}
We introduce the classes
%We partition $\JJ_{U,V,m} \times \JJ_{U,V,m}$ into the classes
\[
\HH_{d,r} :=\{(f,g) \in \JJ_{U,V,m} \times \JJ_{U,V,m}:\,
|\dom f\cap\dom g| = d,
|\ran f\cap\ran g| = r\},
\quad 0 \le d,r \le m,
\]
forming a partition of $\JJ_{U,V,m} \times \JJ_{U,V,m}$
%Recalling, from \eqref{Zfg}, that $Z(f,g)$ is the set of points 
%that have the same images under both $f,g$, we further partition
%each $\HH_{d,r}$ into the classes
and the classes
\[
\HH_{d,r,\ell} := \{(f,g) \in \HH_{d,r}:\, |Z(f,g)|=\ell\}, \quad 
0 \le \ell \le d \wedge r,
\]
forming a partition of $\HH_{d,r}$ for all $0 \le d,r \le m$.
We shall estimate the sizes of these classes. An exact expression
is available for $\HH_{d,r}$. An upper bound for $\HH_{d,r,\ell}$
is sufficient for our purposes.

\begin{lemma}[Cardinality of $\HH_{d,r}$]
\label{CARD1lem}
\begin{equation}
\label{CARD1eq}
|\HH_{d,r}|
= \binom{n}{m-d, m-d, d, n-2m+d} 
\binom{n}{m-r, m-r, r, n-2m+r} 
m!^2.
\end{equation}
\end{lemma}
\begin{proof}
We pick subsets $F$, $G$ of $U$ such that their intersection has $d$ elements.
This can be done in 
\begin{equation} \label{FGpick}
\binom{n}{m-d, m-d, d, n-2m+d}
\end{equation}
ways since we partition
$U$ into the $4$ disjoint sets $F \setminus G$, $G \setminus F$, $F \cap G$, $U
\setminus (F \cup G)$ of sizes $m - d$, $m - d$, $d$, $n - 2m + d$
respectively.  Similarly, there are 
\begin{equation} \label{F1G1pick}
\binom{n}{m-r, m-r, r, n-2m+r}
\end{equation} 
ways to pick $F'$, $G'$ subsets of $V$ such that they have $r$ common elements.
Finally, we can pick $f: F \rightarrow F'$ and $g: G \rightarrow G'$ in $m!^2$
ways.  
\end{proof}

\begin{lemma}[Estimate for the cardinality of $\HH_{d,r,\ell}$]
\label{CARD2lem}
\begin{equation}
\label{CARD2eq}
|\HH_{d,r,\ell}| \le \left| \bigcup_{k=\ell}^{d\wedge r} \HH_{d,r,k} \right|
\le 
|\HH_{d,r}| \binom{d \wedge r}{\ell} \frac{1}{(m)_\ell}.
\end{equation}
\end{lemma}
\begin{proof}
Observe that
$\bigcup_{k=\ell}^{d \wedge r} \HH_{d, r, k}$
is the set of all pairs of functions $(f, g) \in \HH_{d, r}$ such that 
$|Z(f, g)| \ge \ell$. 
Using the following procedure we will ensure that each pair 
$(f, g) \in \bigcup_{k=\ell}^{d \wedge r} \HH_{d, r, k}$ 
is picked at least once.

First, we pick subsets $F$, $G$ of $U$ with $d$ common elements 
and subsets $F'$, $G'$ of $V$ with $r$ common elements. 
Then we pick $f: F \rightarrow F'$ and
a subset $Z$ of $F \cap G$ with $\ell$ elements.  
Finally, we pick some
$\tilde{g}: G \setminus Z \rightarrow G' \setminus f(Z)$ 
which can be extended to $g: G \rightarrow G'$ by setting 
$g(u) = f(u)$ for $u \in G \setminus Z$.
To conclude the proof all we need to do now is to count the ways each step can
be done and multiply them.

The ways to pick $(F,G)$ and $(F',G')$ 
are as in \eqref{FGpick} and \eqref{F1G1pick}, respectively.
The bijection $f: F \to F'$ can be chosen in $m!$ ways.
The cardinality $\ell$ set $Z \subset F \cap G$ can be chosen in $\binom{d}{\ell}$ ways.
The bijection $\tilde{g}: G \setminus Z \rightarrow G' \setminus f(Z)$ 
can be chosen in $(m - \ell)!$ ways.
Multiplying these numbers together and using \eqref{CARD1eq} we obtain that the 
cardinality of $\bigcup_{k=\ell}^{d \wedge r} \HH_{d, r, k}$ is at most
\[
|\HH_{d, r}| \binom{d}{\ell} \frac{1}{(m)_{\ell}}.
\]
By symmetry arguments, we also have that the cardinality of the set is at most
\[
|\HH_{d, r}| \binom{r}{\ell} \frac{1}{(m)_{\ell}}.
\]
The minimum of these two numbers is gives what was claimed in \eqref{CARD2eq}.
\end{proof}

\begin{remark}
The second inequality in \eqref{CARD2eq}
can be further improved, but the improvement will not be used below.
\end{remark}

We need an estimate for the fraction of pairs $(f,g)$ of
partial injections from $U$ to $V$ with domain of size $m$ that
are in $\HH_{d,r}$. 

%\begin{corollary} \label{solon}
%\[
%\color{ForestGreen}
%\frac{|\HH_{d,r,\ell}|}{|\HH_{d,r}|} \le \binom{d \wedge r}{\ell} \frac{1}{(m)_\ell}
%\]
%\end{corollary}
%\begin{proof}
%Directly from Lemma \ref{CARD2lem}.
%\end{proof}

%\subsubsection{Simple asymptotics.}
\begin{lemma} 
\label{CARD3lem}
Fix $d,r$ but let $m$ be a sequence of $n$ such that
$m=m(n) = O(\log n)$ as $n \to \infty$.
There is a sequence $\chi(n)$ such that $\chi(n) \to 1$ and 
\begin{equation}
\label{CARD3eq}
\frac{|\HH_{d,r}|}{|\JJ_{U,V,m}|^2}
= \chi(n)
\binom{m}{d} \binom{m}{r} \frac{(m)_d (m)_r}{n^{d+r}},
\end{equation}
for all $d, r$.
\end{lemma}

\begin{proof}
It follows from \eqref{CARD1eq} 
of Lemma \ref{CARD1lem} and expression \eqref{JJsize} for the size 
of $\JJ_{U,V,m}$,
\[
\frac{|\HH_{d,r}|}{|\JJ_{U,V,m}|^2} 
= h_d h_r,
\]
where
\[
h_x := \frac{\binom{n}{m-x, m-x, x, n-2m+x}}{\binom{n}{m}^2} 
= \frac{(n-m)!^2}{n! (n-2m+x)!} \binom{m}{x} (m)_x.
\]
To show \eqref{CARD3eq} it suffices to show that 
$\frac{(n-m)!^2}{n! (n-2m+x)!} \sim \frac{1}{n^x}$.
We have that
\[
\frac{(n - m)!^2}{n! (n - 2m + x)!} = \frac{(n - m)!^2}{n! (n - 2m)!}
\frac{(n - 2m)!}{(n - 2m + x)!}
\]
where
\[
\frac{(n - m)!^2}{n! (n - 2m)!} =
\prod_{k=1}^m \frac{n-2m+k}{n-m+k}.
\]
Hence we obtain
\[
\left( \frac{n - 2m + 1}{n} \right)^m \le
\frac{(n - m)!^2}{n! (n - 2m)!} \le
\left( \frac{n - m}{n - m + 1} \right)^m.
\]
However, both the upper and lower bounds of $\frac{(n - m)!^2}{n! (n - 2m)!}$
tend to $1$ because $m = O(\log n)$, hence $\frac{(n - m)!^2}{n! (n - 2m)!}
\rightarrow 1$. Moreover, we have that
\[
\frac{(n - 2m)!}{(n - 2m + x)!} = \prod_{k=1}^x \frac{1}{n - 2m + k},
\]
but the latter is bounded by
\[
\frac{1}{n^x} \left(\frac{n}{n - 2m + x}\right)^x \le \frac{(n - 2m)!}{(n - 2m + x)!} \le
\frac{1}{n^x} \left(\frac{n}{n - 2m + 1}\right)^x.
\]
Both the $\left(\frac{n}{n - 2m + x}\right)^x$ and $\left(\frac{n}{n - 2m + 1}\right)^x$ tend to $1$
and hence
\[
\frac{(n - 2m)!}{(n - 2m + x)!} = (1 + o(1)) \frac{1}{n^x}
\]
which concludes the proof.
\end{proof}

\subsubsection{\bfseries A correlation upper bound}
Keeping in mind that we will apply inequality \eqref{CS}, we will 
eventually show that $\E N^2/(\E N)^2 \to 1$ as $n \to \infty$, under the
conditions of Theorem \ref{thm2}.
We have
\[
\frac{\E N^2}{(\E N)^2} =
\frac{1}{|\JJ_{U,V,m}|^{2}} \sum_{f,g \in \JJ_{U,V,m}}
\frac{\E J_f J_g} { \E J_f \E J_g}.
\]
We provide an estimate for the term inside the sum.
Recall that $\tau_{j,k} = p^j q^k+(1-p)^j(1-q)^k$, $\tau=\tau_{1,1}$.

\begin{lemma}
\label{shuplem}
Assume that condition in \eqref{thespis} holds:
\[
\tau^{3/2} > \max(\tau_{1,2}, \tau_{2,1}).
\]
Set
\begin{equation}
\omega := \frac{\max\{pq, (1-p)(1-q)\}}{pq + (1-p)(1-q)}, 
\quad
\beta := \sqrt{\max\left\{\omega, \frac{\tau_{1, 2}\tau_{2, 1}}{\tau^3}\right\}},
\quad
\gamma := \lambda \log (\tau/\tau_{1,2})
\label{gammadef}
\end{equation} 
Then
\[
0 \le \beta < 1, \qquad \frac{1}{2} < \gamma \le 1,
\]
and, if $r \le d$ then,
for all $(f,g) \in \HH_{d,r,\ell}$, $0 \le \ell \le r$,
\begin{equation} 
\label{shupeq}
\frac{\E J_f J_g}{\E J_f \E J_g}
%\le \left( \frac{1}{\tau} \right)^{\binom{d}{2}}
%\left( \frac{\tau_{1,2}}{\tau} \right)^{\binom{d}{2} - \binom{r}{2}} 
%\beta^{\frac{1}{2}(r - \ell)(r-2)}
\le  \left( \frac{1}{\tau} \right)^{(1-\gamma)\binom{d}{2} + \gamma\binom{r}{2}}
\beta^{\frac{1}{2}(r - \ell)(r-2)}.
\end{equation}
\end{lemma}

\begin{proof}
Since $\omega<1$ and
$\tau_{1,2} \tau_{2,1} \le \max(\tau_{1,2}, \tau_{2,1})^2 < \tau^3$. 
we have that $\beta<1$.
Next, by elementary algebra, $\tau^2 \le \tau_{1,2}$. Hence $\tau/\tau_{1/2}  \le 1/\tau$,
and, taking logarithms, $\log(\tau/\tau_{1/2}) \le \log(1/\tau)=1/\lambda$.
This gives $\gamma \le 1$.
Since the condition in \eqref{thespis} holds, $\tau^{3/2} >\tau_{1,2}$. 
Hence $\tau/\tau_{1,2} >1/\tau^{1/2}$.
Taking logarithms we obtain $\log(\tau/\tau_{1,2}) > \frac12 \log(1/\tau)
=\frac{1}{2\lambda}$.
This gives $\gamma >1/2$.
Assume $(f,g) \in \HH_{d,r,\ell}$, which means that
$\dom f \cap \dom g$ has size $d$, $\ran f \cap \ran g$ has size $r$ 
and $Z(f,g)$ has size $\ell$.
We now use the expression 
$\E J_f J_g=\prod_{j,k \ge 1} \tau_{j,k}^{|\CC_{j,k}(f,g)|}$ obtained in
Proposition \ref{N2tris}.
Write $\bar p = 1-p$, $\bar q = 1-q$.
We write 
\begin{align*}
\tau_{j,k} 
&= (pq)^{j-1} p q^{k-j+1} 
+ (\bar p \bar q)^{j-1}  \bar p  \bar q^{k-j+1} 
\\
&\le \max\left\{(pq)^{j-1},\, (\bar p\bar q)^{j-1}\right\}
\, (p q^{k-j+1} + \bar p \bar q^{k-j+1}) 
\\
& = (\omega \tau)^{j-1} \, (p q^{k-j+1} + \bar p \bar q^{k-j+1}).
\end{align*}
We obtain a second inequality by interchanging $j$ and $k$.
We therefore have
\[
\tau_{j,k} \le \begin{cases}
(\omega \tau)^{j-1} \tau_{1,k-j+1} & \text{ if } j \le k
\\
(\omega \tau)^{k-1} \tau_{j-k+1,1} & \text{ if } k \le j
\end{cases}.
\]
Using these inequalities in the expression for $\E J_f J_g$ we have
\begin{align*}
\E J_f J_g \le
\prod_{\substack{j,k \ge 1\\ j \le k}} 
\left[(\omega \tau)^{j-1} \tau_{1,k-j+1}\right]^{|\CC_{j,k}(f,g)|}
\prod_{\substack{j,k \ge 1\\ k<j}} 
\left[(\omega \tau)^{k-1} \tau_{j-k+1,1}\right]^{|\CC_{j,k}(f,g)|}.
\end{align*}
By Lemma \ref{ccut}, only the terms of the form $(j,j+1)$ or $(j,j)$ 
in the first product, and only the terms of the form $(k,k+1)$ in the second product,
survive. Making the change of variable $(j,k) \to (k,j)$ in the second
product and grouping terms together, we obtain
\begin{equation}
\label{once}
\E J_f J_g \le
\omega^{\mathsf A} \, \tau^{\mathsf B}
\, \tau_{1,2}^{\mathsf D_{1,2}}
\, \tau_{2,1}^{\mathsf D_{2,1}},
\end{equation}
where
\begin{align*}
\mathsf A &= \sum_{j \ge 1} 
\big\{(j-1) |\CC_{j,j+1}| + (j-1) |\CC_{j+1,j}| + (j-1) |\CC_{j,j}| \big\}
\\
\mathsf B &= \sum_{j \ge 1} 
\big\{(j-1) |\CC_{j,j+1}| + (j-1) |\CC_{j+1,j}| + j |\CC_{j,j}| \big\}
\\ & \mathsf D_{1,2} = \sum_{j \ge 1} |\CC_{j,j+1}|, \quad
\mathsf D_{2,1} = \sum_{j \ge 1} |\CC_{j+1,j}|
\end{align*}
Using these symbols, \eqref{right} reads
\[
\mathsf B + \mathsf D_{1,2} + 2 \mathsf D_{2,1} = 2 \binom{m}{2}-\binom{d}{2},
\]
and so the right-hand side of \eqref{once} equals
\begin{align}
\omega^{\mathsf A} \tau^{2 \binom{m}{2}-\binom{d}{2}}
\left(\frac{\tau_{1,2}}{\tau}\right)^{\mathsf D_{1,2}} 
\left(\frac{\tau_{2,1}}{\tau^2}\right)^{\mathsf D_{2,1}} 
=
\omega^{\mathsf A} \tau^{2 \binom{m}{2}-\binom{d}{2}}
\left(\frac{\tau_{1,2}}{\tau}\right)^{\mathsf D_{1,2}-\mathsf D_{2,1}}
\left(\frac{\tau_{1,2}\tau_{2,1}}{\tau^3}\right)^{\mathsf D_{2,1}} 
\nonumber
\\
\le \beta^{2 \mathsf A+2 \mathsf D_{1,2}} \tau^{2 \binom{m}{2}-\binom{d}{2}}
\left(\frac{\tau_{1,2}}{\tau}\right)^{\mathsf D_{1,2}-\mathsf D_{2,1}} .
\label{twice}
\end{align}
Rewrite \eqref{por1} and \eqref{por4} as
\[
\mathsf D_{1,2}-\mathsf D_{2,1} = \binom{d}{2}-\binom{r}{2},
\quad
2 \mathsf A+2 \mathsf D_{1,2} \ge \binom{r}{2}-|\ZZ|.
\]
Note that both terms are nonnegative, the first due to the assumption
$r \le d$, and the second due to \eqref{por3}.
We then see that the right-hand side of 
\eqref{twice} is majorized by the right-hand side
of:
\begin{equation}
\label{thrice}
\E J_f J_g 
\le 
\beta^{\binom{r}{2}-|\ZZ|} \tau^{2\binom{m}{2}-\binom{d}{2}}
%\left(\frac{1}{\tau}\right)^{\binom{d}{2}}
\left(\frac{\tau_{1,2}}{\tau}\right)^{\binom{d}{2}-\binom{r}{2}}.
\end{equation}
To conclude the proof, we use the second inequality 
in \eqref{lambdupeq} of Lemma \eqref{lambduplem} to obtain
\[
\binom{r}{2}-|\ZZ| \ge \binom{|Z|}{2}+\frac12 (r-|Z|)
= \binom{r}{2}- \binom{\ell}{2}+\frac12 (r-\ell)
\ge \frac12 (r-\ell)(r-2),
\]
Replacing the exponent of $\beta$ in \eqref{thrice} by the latter
quantity and
dividing both sides by $(\E J_f)(\E J_g) = \tau^{2\binom{m}{2}}$,
we obtain at \eqref{shupeq}.
\[
\frac{\E J_f J_g}{\E J_f \E J_g}
\le \left( \frac{1}{\tau} \right)^{\binom{d}{2}}
\left( \frac{\tau_{1,2}}{\tau} \right)^{\binom{d}{2} - \binom{r}{2}} 
\beta^{\frac{1}{2}(r - \ell)(r-2)}
=  \left( \frac{1}{\tau} \right)^{(1-\gamma)\binom{d}{2} + \gamma\binom{r}{2}},
\]
by the definition of $\gamma$.
\end{proof}

\subsubsection{\bfseries Proof of Phase I of the common subgraph problem}
We aim at proving (I) of Theorem \ref{thm2}. This will be done by showing that
$\limsup_{n \to \infty} \frac{\E N^2}{(\E N)^2}\le 1$, provided that
$m(n)= \lfloor m_*(n)-(C_n/\log n)\rfloor$, $C_n \to \infty$,
$C_n/\log n \to 0$.
Recall that
\begin{equation}
\label{recallratio}
\frac{\E N^2}{(\E N)^2} 
= \sum_{d, r \ge 0} \sum_{(f, g) \in \HH_{d, r}} \frac{\E J_f J_g}{(\E N)^2}
= \frac{1}{|\JJ_{U,V,m}|^2} \sum_{(f, g) \in \JJ_{U,V,m}} \frac{\E J_f J_g}{\E J_f \E J_g}.
\end{equation}
Letting
\begin{equation}
\label{Tdr}
\TTT_{d, r} 
= \frac{1}{|\JJ_{U,V,m}|^2} 
\sum_{(f, g) \in \HH_{d, r}} \frac{\E J_f J_g}{\E J_f \E J_g}, \quad
0 \le d \le r \le m,
\end{equation}
we write, for some appropriate $0<c<1$,
\begin{gather}
\frac{\E N^2}{(\E N)^2} %= \sum_{d,r} \TTT_{d,r}
= \TTT_{0,0} + \TTT_{m,m} + \sum_{\substack{r\le d,\,r\le cm \\ (d,r)\neq (0,0) }} \TTT_{d,r}
+ \sum_{\substack{cm < r \le d \\ (d,r) \neq (m,m) }} \TTT_{d,r}
+ \sum_{r < d} \TTT_{d,r}.
\label{saksa}
\end{gather}  
Each of the five terms on the right of \eqref{saksa} will be treated separately.

\begin{lemma} \label{termuplem}
Let $\TTT_{d,r}$ be defined by \eqref{Tdr} and $\gamma$ by \eqref{gammadef}
There are universal sequences $\chi(n) \to 1$, $\psi(n)\to 1$ such that:
\\
If $0\le r \le d\le m$ we have
\begin{equation}
\label{termupeq1}
\mathcal \TTT_{d, r} \le \chi (n) \binom{m}{d} \binom{m}{r} \frac{(m)_d (m)_r}{n^{d+r}}
\left( \frac{1}{\tau} \right)^{(1-\gamma) \binom{d}{2} + \gamma \binom{r}{2}}.
\end{equation}
If, in addition, $cm \le r \le d$, for some $0<c<1$, we have the sharper inequality
\begin{equation}
\label{termupeq2}
\mathcal \TTT_{d, r} \le \chi(n) \psi (m) \binom{m}{d} \binom{m}{r} \frac{(m)_d}{n^{d+r}}
\left( \frac{1}{\tau} \right)^{(1-\gamma) \binom{d}{2} + \gamma \binom{r}{2}}.
\end{equation}
\end{lemma}

\begin{proof}
Assume $r \le d$.
Look at the expression \eqref{Tdr} for $\TTT_{d,r}$.
Using \eqref{shupeq} of Lemma \ref{shuplem} we have
%upper-bound the term inside the sum of \eqref{Tdr} as follows:
\[
\frac{\E J_f J_g}{\E J_f \E J_g}
\le  \left( \frac{1}{\tau} \right)^{(1-\gamma)\binom{d}{2} + \gamma\binom{r}{2}}            
\beta^{\frac{1}{2}(r - \ell)(r-2)}
\le  \left( \frac{1}{\tau} \right)^{(1-\gamma)\binom{d}{2} + \gamma\binom{r}{2}},
\]
where we used the fact that $\beta \le 1$.
Inserting this in \eqref{Tdr} we have
\[
\TTT_{d, r} 
%\le \frac{\left|\HH_{d, r}\right|}{\left|\JJ_{U, V,m}\right|^2} \left( \frac{1}{\tau} \right)^{\binom{d}{2}}
%\left( \frac{\tau_{1,2}}{\tau} \right)^{\binom{d}{2} - \binom{r}{2}} 
\le \frac{\left|\HH_{d, r}\right|}{\left|\JJ_{U, V,m}\right|^2}
\left( \frac{1}{\tau} \right)^{(1-\gamma)\binom{d}{2} + \gamma\binom{r}{2}}.
\]
Then using \eqref{CARD3eq} of Lemma \ref{CARD3lem} we further have that
\[
\TTT_{d,r} \le \chi (n) \binom{m}{d} \binom{m}{r} \frac{(m)_d (m)_r}{n^{d+r}}
\left( \frac{1}{\tau} \right)^{(1-\gamma) \binom{d}{2} + \gamma \binom{r}{2}},
\]
where $\chi(n) \to 1$ and $\chi(n)$ does not depend on $d$ or $r$,
proving \eqref{termupeq1}.

Assume further that $cm \le r \le d$.
Using that $\HH_{d,r}$ is the disjoint union of $\HH_{d,r,\ell}$, $\ell=0,\ldots, r$,
further split the sum in \eqref{Tdr} and then
%%%%%%%%%%%   \tkside{Tidy up a bit!}
{use inequality \eqref{shupeq} of Lemma \ref{shuplem} again} in full force
(we do not omit the $\beta$) to obtain 
\begin{align*}
\TTT_{d, r} &= \frac{1}{|\JJ_{U,V,m}|^2} 
\sum_{\ell=0}^r \sum_{(f, g) \in \HH_{d, r,\ell}} 
\frac{\E J_f J_g}{\E J_f \E J_g}
\\
&\le 
\sum_{\ell=0}^r \frac{1}{|\JJ_{U,V,m}|^2} \sum_{(f, g) \in \HH_{d, r,\ell}} 
\left( \frac{1}{\tau} \right)^{(1-\gamma)\binom{d}{2} + \gamma\binom{r}{2}}
\beta^{\frac{1}{2}(r - \ell)(r-2)}
\\
&= 
\left( \frac{1}{\tau} \right)^{(1-\gamma)\binom{d}{2} + \gamma\binom{r}{2}}
\sum_{\ell=0}^r 
\frac{|\HH_{d,r,\ell}|}{|\JJ_{U,V,m}|^2}
\,\beta^{\frac{1}{2}(r - \ell)(r-2)}
\end{align*}
From \eqref{CARD2eq} of Lemma \ref{CARD2lem} we have (since $\ell \le r \le d$)
\[
|\HH_{d,r,\ell}| \le {|\HH_{d,r}|} \binom{r}{\ell} \frac{1}{(m)_\ell}
\le  {|\HH_{d,r}|} \binom{r}{\ell} \frac{m^{r-\ell}}{(m)_r}.
\]
So we further obtain
\begin{equation}
\label{further}
\TTT_{d, r} \le \left( \frac{1}{\tau} \right)^{(1-\gamma)\binom{d}{2} + \gamma\binom{r}{2}}
\frac{1}{(m)_r}
\frac{|\HH_{d,r}|}{|\JJ_{U,V,m}|^2}
\, \sum_{\ell=0}^r 
\binom{r}{\ell} m^{r-\ell}
\beta^{\frac{1}{2}(r - \ell)(r-2)}
\end{equation}
We now use the assumption that $r>cm$ to get
$\beta^{\frac{1}{2}(r - \ell)(r-2)} \le \beta^{\frac{1}{2}(r - \ell)(cm-2)}$
so that
\[
\sum_{\ell=0}^r 
\binom{r}{\ell} m^{r-\ell}
\beta^{\frac{1}{2}(r - \ell)(r-2)}
\le \left(1+m \beta^{(cm-2)/2}\right)^m := \psi(m),
\]
noting that $\psi(m) \to 1$ as $m$ tends to $\infty$.
We now replace the last sum of \eqref{further} by $\psi(m)$
and the term
$\frac{1}{(m)_r}
\frac{|\HH_{d,r}|}{|\JJ_{U,V,m}|^2}$
by 
$\chi(n) \binom{m}{d} \binom{m}{r} \frac{(m)_d}{n^{d+r}}$ from \eqref{CARD3eq}
to immediately arrive at \eqref{termupeq2}.
\end{proof}

\begin{proposition}
\label{allterms}
Assume that $(p, q) \in \Y$, as in \eqref{thespis}.
Take $m=m(n) = \lfloor m_*(n) - C_n/\log n\rfloor$, 
with $C_n \to \infty$, $C_n/\log n \to 0$,
as $n \to \infty$.
%With reference to \eqref{saksa}, for fixed $c$, with $1/2<c<1$, we have
Fix a constant $c$ strictly between $0$ and $1$. Then
\begin{enumerate}[\rm (i)]
\item
$\limsup_{n \to \infty} \TTT_{0,0} \le 1;$
\item
$\lim_{n \to \infty} \TTT_{m,m}= 0;$
\item
$\lim_{n \to \infty} \displaystyle \sum_{\substack{r\le d,\,r\le cm \\ (d,r)\neq (0,0) }} \TTT_{d,r} = 0;$
\item
$\lim_{n \to \infty} \displaystyle 
\sum_{\substack{cm \le r \le d \\ (d,r) \neq (m,m) }} \TTT_{d,r} = 0;$
\item
$\lim_{n \to \infty} \displaystyle \sum_{\substack{d \le r\\(d,r) \neq(0,0)}} 
\TTT_{d,r} = 0.$
\end{enumerate}
\end{proposition}

\begin{proof}
(i) 
Setting $d=r=0$ in \eqref{termupeq1} of
Lemma \ref{termuplem} for some $c \in (0, 1)$ we obtain
\[
\TTT_{0, 0} \le \chi(n) \to 1.
\]

(ii) 
Setting $d=r=m$ in \eqref{termupeq2} we obtain
\begin{equation}
\label{Tmm}
%\TTT_{m, m} \le \xi(n) \frac{m!}{n^{2m}} \left(\frac{1}{\tau}\right)^{\binom{m}{2}} 
\TTT_{m, m} \le \xi(n) \psi(m) \frac{m!}{n^{2m}} \left(\frac{1}{\tau}\right)^{\binom{m}{2}} 
= \xi(n) \psi(m)  \frac{(n)_m^2}{n^{2m}} \frac{1}{\E N},
\end{equation}
where we used expression \eqref{crito} for $\E N$.
We now recycle arguments from the proof of Lemma \ref{ENasymptlem}.
From \eqref{ENasympt}, we have that 
\begin{equation}
\label{Nb}
\E N/b(n) \to 1, 
\end{equation}
where 
\[
b(n)= \frac{n^{2m}}{(m/e)^m \sqrt{2\pi m}} \tau^{\frac12 m(m-1)},
\] 
as defined in \eqref{ENasympt}.
The same little algebra that led to \eqref{weealg} gives
\[
\log b(n) = \frac{1}{2\lambda} \, m\, ( W(m_*) - W(m) ),
\]
but now with $m= \lfloor m_* - C_n/\log n\rfloor \le m_*$.
We argued in \eqref{Wprime} that $W'(x) \ge 1$ for all $x>0$, therefore,
\[
\log b(n) \ge \frac{1}{2\lambda} \,m \, (m_*-m)
\ge \frac{1}{2\lambda} \,m \, \frac{C_n}{\log n}.
\]
We now have
\[
\lim_{n \to \infty} \frac{m}{\log n}
= \lim_{n \to \infty} \frac{m_*}{\log n} = 4\lambda,
\]
as argued in \eqref{4lambda}.
Therefore $\log b(n) \to \infty$ and so, by \eqref{Nb},  $\E N \to \infty$.
We now look at the right-hand side of \eqref{Tmm}, realizing that the first three 
terms converge to $1$, while the last converges to $0$.
Hence $\lim_{n \to \infty} \TTT_{m,m}=0$, as claimed.

\begin{figure}[H]
\begin{center}
\includegraphics[width=0.38\textwidth]{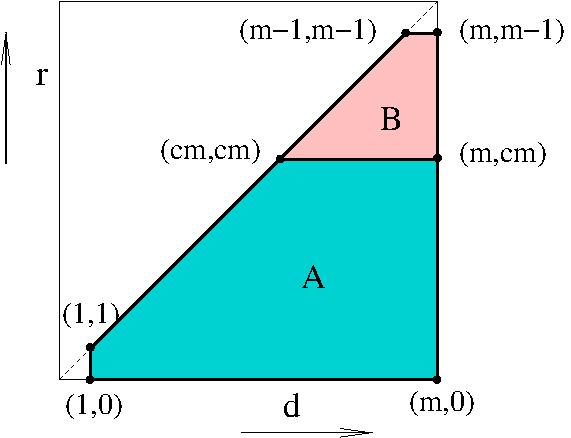}
\captionof{figure}{The sets  $A$, $B$ used in (iii), (iv), respectively.}
\label{tri}
\end{center}
\end{figure}

(iii) 
%Assume next that $r \le \min(d, cm)$ but $(d,r) \neq (0,0)$.
With $0 \le r \le d \le m$, \eqref{termupeq1} holds.
Using the upper bounds $\binom{m}{d} \le m^d$, $(m)_d \le m^d$,
$\binom{m}{r} \le m^r$, $(m)_r \le m^r$ in \eqref{termupeq1} we obtain
\[
\TTT_{d, r} \le \chi(n) \left(\frac{m^2}{n}\right)^{d+r}
\left( \frac{1}{\tau} \right)^{(1 - \gamma) \binom{d}{2} + \gamma \binom{r}{2}},
\]
and so
\[
\log(\TTT_{d, r}) \le \log\chi(n) +  P(d,r),
%(d + r)(2 \log(m) - \log(n)) +
%\frac{1 - \gamma}{\lambda} \binom{d}{2} + \frac{\gamma}{\lambda} \binom{r}{2}.
\]
where
\begin{equation}
\label{Pxy}
P(x,y) := (x + y)(2 \log(m) - \log(n)) + 
\frac{1 - \gamma}{\lambda} \binom{x}{2} + \frac{\gamma}{\lambda} \binom{y}{2}.
\end{equation}
Note that $P(x,y)$ is a convex quadratic polynomial. We seek to estimate the 
sum of $\TTT_{d,r}$ over pairs of integers $(d,r)$ 
such that $r \le d$, $r \le cm$ and $(d,r) \neq (0,0)$.
If we let $A \subset \R^2$ be the convex hull (Fig.\ \ref{tri}) of the points
\[
\text{$(1, 0)$, $(1, 1)$, $(m, 0)$, $(m, cm)$, $(cm, cm)$,}
\]
then the sum is over pairs of integers in $A$. We seek to find 
an asymptotic upper bound of $P$ on $A$.
Since $A$ is closed and convex, the maximum of $P$ 
is achieved at one of the above 5 points.
We have
\[
P(1, 0) = 2 \log m - \log n = -\frac{1}{4 \lambda} m +O(\log m)  ;
\]
\[
P(1, 1) = 4 \log m - 2 \log n = -\frac{1}{2 \lambda} m +O(\log m)  ;
\]
\[
P(m, 0) = 2 m \log m - m \log n + \frac{1 - \gamma}{2\lambda} m(m-1)
= 
-\frac{2 \gamma - 1}{2 \lambda} m^2 +  O(m \log m)   ;
\]
\[
P(cm, cm) = 4 cm \log m - 2cm \log n + \frac{1}{2\lambda} cm(cm-1)
= -\frac{c - c^2}{2 \lambda} m^2 + O(m \log m)   ;
\]
\begin{multline*}
P(m, cm) = 2 (1 + c) m \log m - (1 + c) m \log n
+ \frac{1 - \gamma}{2\lambda} m(m-1)  + \frac{\gamma}{2\lambda} cm(cm-1)
%\\
%&= O(m \log m) - \frac{1}{4 \lambda} m^2 ((1+c) - 2 (1-\gamma) - 2c^2 \gamma) 
\\
= -\frac{(1 - c) (2 \gamma (1 + c) - 1)}{4 \lambda} m^2 +  O(m \log m) .
\end{multline*}
Indeed, the first equality in each of these expressions follows 
directly from \eqref{Pxy}, while the second one follows by letting
$\log n = m/4\lambda + O(\log m)$, as in Corollary \ref{coroA2}.
Examining the leading terms in the above expressions, i.e., the terms
of order $m$ for the first two and order $m^2$ for the last three,
we see that all coefficients are negative. (Here we need that $\gamma>1/2$.)
The largest of them is therefore of order $m$. 
Comparing the coefficients of order $m$ in the first two expressions, we
conclude that
\[
\max_{(x,y)\in A} P(x,y) = P(1, 0) = 2 \log m - \log n, \quad \text{ for all large $n$.}
\]
%Comparing the values of $P$ at this 5 points, we see,
%using the fact that $\gamma > 1/2$ (see Lemma \ref{shuplem}) that the maximum is achieved
%\tkside{$\checkmark$ Check $\red \gamma > 1/2$ is what is needed.} 
%at the point $(1,0)$ and is equal to
\iffalse
\footnote{\blue
Hello
$$W(m^*) =  4 \lambda \log n + 2 \lambda+1$$
$$W'(1) = 1+\lambda(3- \log(2\pi) > 0$$
$$W(m_*) - W(m) \le W'(1) (1+C_n/\log n) \le B$$
$$W(m)-m-2\lambda \log m = \frac{\lambda \log(2\pi m)}{m} \to 0 $$.
$$4 \lambda \log n - m
= [W(m_*) - W(m)] - 2 \lambda -1 + [W(m) -m]
\le B  - 2 \lambda -1 + 2\lambda \log m + o(1)
= O(\log m)
$$
}
\fi
%\tkside{$\checkmark$ Copy here proof for this from above Lemma \ref{ENasymptlem} }
Hence
\[
\sum_{\substack{r\le d,\,r\le cm \\ (d,r)\neq (0,0) }} \TTT_{d,r} 
\le \chi(n) \sum_{\substack{r\le d,\,r\le cm \\ (d,r)\neq (0,0) }} \frac{m^2}{n}
\le \chi(n) \frac{m^4}{n} \to 0,
\]
where we used Corollary \ref{coroA2} again for the last convergence to zero.

(iv)  
Since we are now interested in the case $cm \le r \le d$,
we use the more detailed estimate  \eqref{termupeq2}
for $\TTT_{d,r}$ 
Using the bounds $\binom{m}{d} \le m^{m-d}$, $\binom{m}{r} \le m^{m-r}$,
$(m)_d \le m!$ in \eqref{termupeq2} we obtain
\[
\TTT_{d, r} \le \chi(n) \psi (m)  \frac{m^{2m - d - r} m!}{n^{d + r}}
\left( \frac{1}{\tau} \right)^{(1-\gamma)\binom{d}{2} + \gamma \binom{r}{2}}.
\]
Taking logarithms and writing $\log m! \le m \log m -m + C$, for some constant $C$, we have
\begin{equation}
\label{logT}
\log (\TTT_{d, r}) \le \log(\xi(n) \psi(m)) + O(\log m) + Q(d, r),
\end{equation}
where
\[
Q(x, y) = (3m - x - y) \log m 
- m - (x + y) \log n  + \frac{1 - \gamma}{\lambda} \binom{x}{2} 
+ \frac{\gamma}{\lambda} \binom{y}{2}.
\]
As before, we to find an asymptotic upper bound for
$Q(x,y)$, but now over the convex hull $B\subset \R^2$  
(Fig.\ \ref{tri}) of the points 
\[
\text{$(m,m-1)$, $(m-1, m-1)$, $(cm, cm)$, $(m, cm)$}.
\]
Again, $\max_{(x,y) \in B} Q(x,y)$ is achieved at one
of these 5 points.
%{\tiny Using Corollary \ref{coroA3} we obtain
%$\log n > \frac{1}{4 \lambda} m + \frac{1}{2} \log m - \frac{1}{2} -\frac{1}{4 \lambda} + o(1)$
%}
We have
\[
Q(m, m-1) = -m + (m + 1) \log m - (2m - 1) \log n
+ \frac{1-\gamma}{2\lambda}m(m-1) + \frac{\gamma}{2\lambda} (m-1)(m-2).
\]
Replacing $\log n$ by the asymptotic lower bound of \ref{coroA3}
we obtain a linear combination of $m^2$, $m \log m$, $m$, and $o(m)$.
The coefficients of the first two terms vanish and we are left with
\[
Q(m,m-1) \le - {\frac{4 \gamma - 1}{4 \lambda} m} +o(m),\,
\text{ for all large $n$}.
\]
For the second point,
\[
Q(m-1, m-1)  = - m + (m + 2) \log m  - 2 (m - 1) \log n 
+ \frac{1}{2\lambda} m(m-1).
\]
Lower bounding $\log n$ in the same way, terms involving $m^2$ and $m\log m$
are annihilated and we obtain
\[
Q(m,m-1) \le {-\frac{1}{2 \lambda} m} + o(m),\,
\text{ for all large $n$}.
\]
For the next  two extreme points of $B$ no terms vanish and we obtain,
for all large $n$,
\begin{align*}
Q(cm, cm)  %= -m + (3 - 2c) m \log m - 2 cm \log n 
%+ \frac{1}{2\lambda} cm (cm-1) 
&\le - {\frac{c(1-c)}{2 \lambda} m^2} + O(m \log m)
\\
Q(m, cm) 
&\le -{\frac{(1 - c) ( 2\gamma (1+c) - 1)}{4 \lambda} m^2}+O(m \log m).
\end{align*}
We conclude that
\[
\max_{x,y} Q(x,y) 
\le - \min\left\{\frac{4 \gamma - 1}{4 \lambda},\,\frac{1}{2\lambda}\right\} m
+o(m) \le - C_1 m,
\, \text{ for all large $n$,}
\]
for some $C_1>0$. Hence, from \eqref{logT}, there is $C_2>0$  such that
\[
\log(\TTT_{d,r}) \le - C_2 m, \, \text{ for all large $n$.}
\]
We finally obtain that
\[
\sum_{\substack{cm\le r\le d\\(d,r)\neq(0,0)}} \TTT_{d,r} 
\le m^2  e^{-C_2 m} \to 0.
\]

(v) 
The facts proven in (ii), (iii) and (iv) imply that
$\sum_{r \le d, (d,r) \neq (0,0)} \TTT_{d,r} \to 0$.
Symmetry arguments imply that 
$\sum_{d \le r, (d,r) \neq (0,0)} \TTT_{d,r} \to 0$ as well.
For the symmetry arguments we need to interchange $p$ and $q$ and
replace the $\gamma$ of \eqref{gammadef} by 
$\tilde \gamma=\log (\tau/\tau_{2,1})$.
The admissible region of $(p,q)$ specified by 
$\max(\tau_{1,2}, \tau_{2,1}) < \tau^{3/2}$ is symmetric in $p$ and $q$.
So $\gamma > 1/2$ implies that $\tilde \gamma>1/2$ as well.
\end{proof}

\begin{proof}[Concluding the proof of Theorem \ref{thm2}]
Using Proposition \ref{allterms} and \eqref{saksa} we have
that $\limsup \E N^2/(\E N)^2 \le 1$.
Since $\limsup (1/\P(N>0)) \le \limsup \E N^2/(\E N)^2 \le 1$, we conclude that
$\P(N>0) \to 1$, as claimed.
\end{proof}

%which concludes the proof of theorem \ref{thm2}.

\section{\bfseries Further remarks}
The following characterization of the admissible region $\Y$,
defined in \eqref{thespis} is worth pointing out.
\begin{proposition}
\label{propc}
Let $n \ge 2$ be an integer.
Set
\[
m \equiv m(n) = \lfloor m_*(n)-(C_n/\log n)\rfloor,
\]
where $C_n \to \infty$ such that $C_n/\log n \to 0$,
and let $N\equiv N(n,p,q)$ be the number of $m$-isomorphisms between two independent
$G(n,p)$ and $G(n,q)$ random graphs. 
Then
\begin{align*}
\Y &= \left\{(p,q) \in (0,1)\times (0,1): \, \lim_{n \to \infty}
\E N^2/(\E N)^2 = 1\right\},
\\
\Y^c &= \left\{(p,q) \in (0,1)\times (0,1): \, \lim_{n \to \infty}
\E N^2/(\E N)^2 = \infty\right\}.
\end{align*}
\end{proposition}
\begin{proof}
Inequality \eqref{saksa} says that $\E N^2/(\E N)^2$ is bounded by a quantity
which, by Proposition \ref{allterms},  
has limsup $\le 1$. Since $\E N^2/(\E N)^2 \ge 1$,
it follows that
$\Y \subset \left\{(p,q) : \, \lim_{n \to \infty} \E N^2/(\E N)^2 = 1\right\}$.
Once we show that
\begin{equation}
\label{Yc}
\Y^c \subset 
\left\{(p,q) : \, \lim_{n \to \infty} \E N^2/(\E N)^2 = \infty \right\}
\end{equation}
the proof will be complete.
Suppose that $(p,q) \in \Y^c$, that is, $\max\{\tau_{1,2}(p,q),
\tau_{2,1}(p,q)\} \ge \tau(p,q)^{3/2}$.
Without loss of generality, assume $\tau_{1,2}	 \ge \tau^{3/2}$.
(We omit the dependence on $p,q$ from the notation.)
We look at the expression \eqref{recallratio} for $\E N^2/(\E N)^2$
and set $d=m, r=0$ on the right-hand side to obtain a lower bound:
\[
\frac{\E N^2}{(\E N)^2} 
\ge \sum_{(f,g) \in \HH_{m,0}} \frac{ \E J_f J_g}{(\E N)^2}
= \frac{1}{|\JJ_{U,V,m}|^2} \sum_{(f,g) \in \HH_{m,0}} 
\frac{\E J_f J_g}{\tau^{2 \binom{m}{2}}},
\]
where we used $\E N = |\JJ_{U,V,m}| \tau^{\binom{m}{2}}$; see 
\eqref{crito}.
It is easy to see that $\E J_f J_g$ is the same for all $(f,g) \in \HH_{m,0}$.
Indeed, recalling the definition of $\HH_{d,r}$ as the 
set of pairs $(f,g)$ of partial injections such
that their domains overlap on $d$ points and their ranges on $r$ points,
we have that $(f,g) \in \HH_{m,0}$ iff $f, g$ have a common domain,
say $D$, of size $m$, and disjoint ranges. Hence
\[
\E J_f J_g 
= \P(X(e) = Y(f(e)) = Y(g(e)) \text{ for all } e \in \PP_2(D))
= \tau_{1,2}^{\binom{m}{2}},
\]  
because $\PP_2(D)$ has $\binom{m}{2}$ elements,
and the random variables $X(e)$, $Y(f(e)$, $Y(g(e))$
are Bernoulli with parameters $p$, $q$, $q$, respectively;
moreover, independence is guaranteed since 
$e, f(e), g(e)$ range ever the pairwise disjoint sets $D$, $f(D)$, $g(D)$,
respectively.
Therefore,
\[
\frac{\E N^2}{(\E N)^2}  
\ge \frac{|\HH_{m,0}|}{|\JJ_{U,V,m}|^2}
\left(\frac{\tau_{1,2}}{\tau^2}\right)^{\binom{m}{2}}
\ge \frac{|\HH_{m,0}|}{|\JJ_{U,V,m}|^2}\,
\frac{1}{\tau^{\frac12\binom{m}{2}}},
\]
where we used the assumption $\tau_{1,2}   \ge \tau^{3/2}$
to obtain the last inequality.
Since $m=O(\log n)$, Lemma \ref{CARD3lem} applies.
With $d=m$ and $r=0$, \eqref{CARD3eq} reads
\[
\frac{|\HH_{m,0}|}{|\JJ_{U,V,m}|^2} = \chi(n) \frac{m!}{n^m},
\]
for some universal sequence $\chi(n)$ such that $\chi(n)\to 1$
as $n \to \infty$.
Hence
\begin{equation}
\label{rlb}
\frac{\E N^2}{(\E N)^2}  \ge \chi(n) \frac{m!}{n^m \tau^{\frac12\binom{m}{2}}}
\sim 
\chi(n) \sqrt{m!} \sqrt{\frac{(m/e)^m \sqrt{2\pi m}}{n^{2m}\tau^{\binom{m}{2}}}}
= \chi(n) \sqrt{m!} 
\sqrt{\frac{1}{b(n)}},
\end{equation}
where $b(n)$ is defined by \eqref{rlb} itself and is the same quantity
appearing in the proof of Lemma \ref{ENasymptlem}.
Hence, as in \eqref{weealg},
\[
\log b(n) = \frac{1}{2\lambda} \, m\, ( W(m_*) - W(m) ).
\]
Using the mean value theorem and the fact that $W'(x) \le W'(1)$ for
all $x \ge 1$, we have
\[
\log b(n) \le \frac{W'(1)}{2\lambda} m (m_*-m).
\]
But 
\[
m=\lfloor m_*-(C_n/\log n)\rfloor > m_*-(C_n/\log n)-1,
\]
hence
\[
\log b(n) \le \frac{W'(1)}{2\lambda} m(1+C_n/\log n),
\]
and this implies that $b(n) \le e^{c_1 m}$, eventually,
for some positive constant $c_1$. On the other hand 
$m! \ge e^{c_2 m \log m}$ for some positive constant $c_2$.
So the right side of \eqref{rlb} tends to $\infty$ as $n \to \infty$
and so \eqref{Yc} holds.
\end{proof}

\iffalse
\begin{remark}
Proposition \ref{propc} actually shows that some of the proofs
in the recent paper by Surya {\em et al.} \cite{SWZ2023} 
%[Theorem 3]
do not work.
To see this concretely, take, for instance, Theorem 3 of
this paper where it is claimed that a phase transition for
the random subgraph isomorphism problem between $G(n,p)$ and $G(n,q)$ 
holds true for all values of $p$ and $q$.
The authors' arguments are based on a heuristic:
since working with $G(n, p), G(n, q)$ is hard,
the authors replace them by what they call ``pseudorandom'' graphs,
namely they assume that certain {\em a posteriori} properties
hold.
In particular, in their proof (\cite[\S3.2]{SWZ2023}), they use
a statement that, {\red in our notation,
translates into $\var(N) = o((\E N)^2)$, as $n \to \infty$,}
{\em for all } $p,q$. But this is equivalent to $\E N^2/(\E N)^2 \to 1$
for all $p,q$. However, Proposition \ref{propc} shows that
$\E N^2/(\E N)^2 \to \infty$ when $(p,q) \not\in \Y$.
Part of the problem is that delicate questions like 
these can be vastly different between a $G(n,p)$ model
and a model where (a) the random graph is chosen uniformly at
random from all graphs on $n$ vertices and a 
number of edges that is between two non-random bounds
and (b) on further conditioning on certain desirable properties.
Hence the phase transition problem, for arbitrary $p$ and $q$,
remains open. See Open Problem P\ref{outside} below.
\end{remark}
\fi

\section{\bfseries Some open problems}
Here is a (partial) list of open problems.
\begin{enumerate}[P1.]
\item
The graph embedding phase transition problem between
$G(m,p)$ and an independent $G(n,q)$ has been fully solved
when $q=1/2$. The case for general $q$ remains open.

\item
As mentioned in Remark \ref{Wrem}(v), there is a delicate question involving
graph embedding of $G(2k+1,p)$ and $G(2^k, 1/2)$; this is an open question
that remains delicate even when $p=1/2$.

\item
Another open question is the study of the distribution of an appropriately
normalized version of $N$ 
(for both the embeddability and the common subgraph problems)
so that it converges to some limit.

\item
How does a largest common induced subgraph between $G(n,p)$ and $G(n,q)$
evolve as $n$ increases? This is a question also asked in \cite{ChD23}.

\item
In the common subgraph problem we considered both graphs to
have the same size $n$. But what happens when one has size $n_1$ and
the other $n_2$, both tending to infinity in a particular way?

\item
An isomorphism between colored graphs can also be defined. One
can thus ask similar questions in this case. The parameter here 
is the whole distribution of the colors in each graph.

\item
Recently, Lenoir \cite{L2024} addressed the common subgraph
problem between two random $d$-hypergraphs, but only
in the uniform case (that is, when we put the uniform
probability measure on the set of all $d$--hypergraphs on $n$ vertices).
For $d=2$ this corresponds to the case $p=q=1/2$. 
One can ask the question of phase transition for more general
than uniform distributions as well.

\item
\label{outside}
As explained, the region $\Y$ is only sufficient, but not necessary,
for the existence of phase transition.
See also Proposition \ref{propc}.
Outside $\Y$ the bound \eqref{CS} does not give useful information.
So the problem is to investigate what is going on outside $\Y$,
being precisely the region where the second moment method does
not provide any information.

\item
Finally one can ask same questions when the parameters. e.g.,
$p$ and $q$ depend on $n$. (For example, phase transition
for the chromatic number of sparse graphs $G(n,d/n)$ is proved in
\cite{AN2005}.)
\end{enumerate}

\renewcommand\refname{\bfseries References}
\bibliographystyle{amsplain}

\appendix
\section{\bfseries Supplementary information}
\label{suppinfo}

\noindent{\bf Probabilities of coincidence of sets of edges.}
\begin{align*}
\tau_{j,k} &= p^jq^k + (1-p)^j (1-q)^k
\\
\tau_{1,2} &= pq^2 + (1-p) (1-q)^2
\\
\tau &= pq+ (1-p)(1-q) = \tau_{1,1} 
\end{align*}

\noindent{\bf Parameters.}
\begin{align*}
\lambda &= 1/\log(1/\tau)
\\
\gamma &= \lambda \log (\tau/\tau_{1,2})
\end{align*}

\noindent{\bf The admissible region.}
\[
\Y:=\{(p,q) \in (0,1)\times (0,1):\, \max(\tau_{1,2}, \tau_{2,1}) < \tau^{3/2}\}
\]

\noindent
\begin{tabular}{cccc}
\begin{minipage}{0.23\textwidth}
\begin{figure}[H]
\centering
\includegraphics[width=\textwidth]{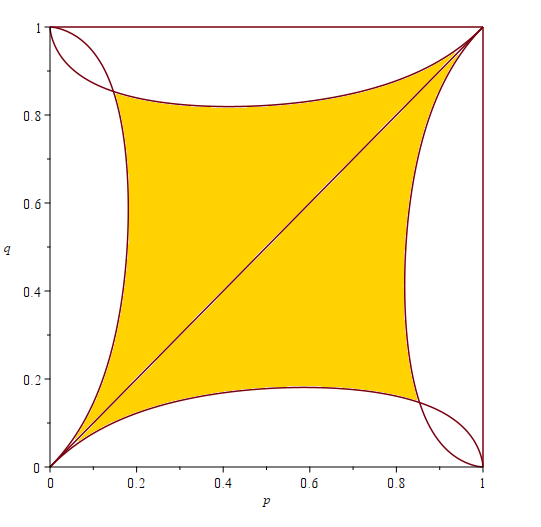}
\end{figure}
\end{minipage}
&
\begin{minipage}{0.23\textwidth}
\begin{figure}[H]                                                                 
\centering
\includegraphics[width=0.9\textwidth]{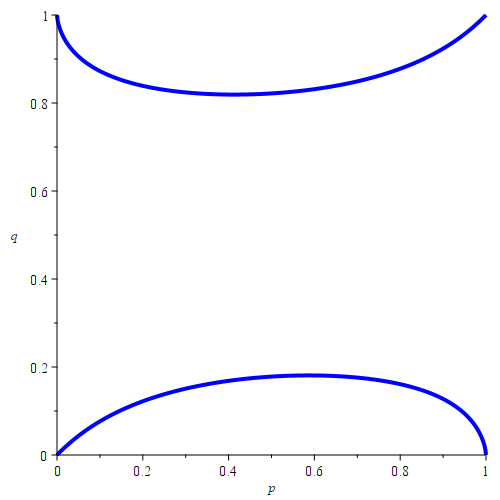}
\end{figure}
\end{minipage}
&
\begin{minipage}{0.23\textwidth}
\begin{figure}[H]                                                                 
\centering
\includegraphics[width=0.9\textwidth]{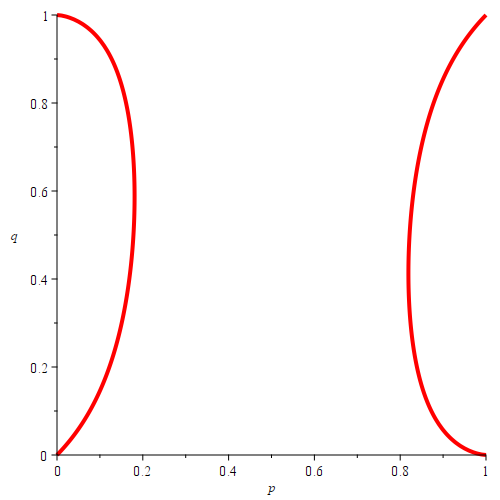}
\end{figure}
\end{minipage}
&
\begin{minipage}{0.23\textwidth}
\begin{figure}[H]                                                                 
\centering
\includegraphics[width=0.9\textwidth]{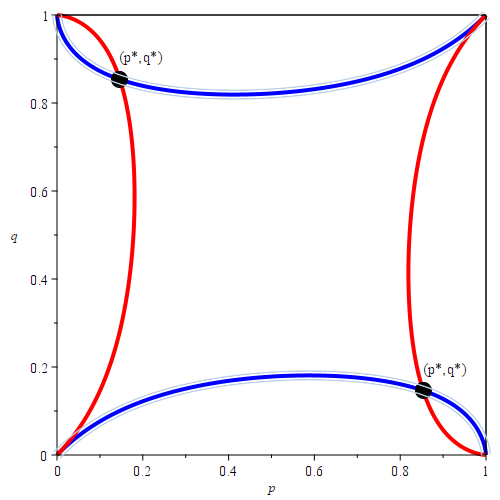}
\end{figure}
\end{minipage}
\\
The set $\Y$
&
$\tau_{1,2}(p,q)=\tau(p,q)^{3/2}$
&
$\tau_{2,1}(p,q)=\tau(p,q)^{3/2}$
&
$(p^*,q^*)$ and $(q^*,p^*)$
\end{tabular}

\noindent{\bf Corners of the admissible region.}
The point $(p^*, q^*)$ is defined as the unique solution of
$\tau_{1,2}(p,q) = \tau(p,q)^{3/2}$. We have
$p^* \approx 0.1464466094$, $q^*\approx 0.8535533906$.

\noindent{\bf Extrema of $\lambda$.}
We have
\begin{align*}
\min_{(p,q)\in [0,1]\times[0,1]} \lambda(p,q) &= \lambda(0,1)=\lambda(1,0)=0
\\
\min_{(p,q)\in \Y} \lambda(p,q) &= \lambda(p^*, q^*) = \lambda(q^*,p^*)
\approx 0.7213475205
\\
\sup_{(p,q)\in \Y} \lambda(p,q) &= \lim_{(p,q)\to (1,1)} \lambda(p,q)
= \lim_{(p,q)\to(0,0)} \lambda(p,q) = \infty.
\end{align*}

\noindent{\bf Extrema of $\gamma$}
\begin{align*}
\min_{(p,q)\in [0,1]^2} \gamma(p,q) &= 0 = \gamma(x,0) = \gamma(x,1), \quad 0<x<1
\\
\max_{(p,q)\in [0,1]^2} \gamma(p,q) &= 1 = \gamma(0,y) = \gamma(1,q), \quad 0<y<1,
\\
\gamma(x,1/2) &=1 , \quad 0 \le x \le 1
\\
\min_{(p,q)\in \Y} \gamma(p,q) &= \gamma(p^*,q^*) = \gamma(q^*,p^*) = 1/2
\\
\max_{(p,q)\in \Y} \gamma(p,q) &= \gamma(1/2, 1/2)=1
\end{align*}

\noindent{\bf The function $W$.} 
\[
W(x) := x+ 2\lambda \log x + \frac{\lambda}{x} \log(2\pi x), \quad x \ge 1.                
\]
It is strictly increasing because, for all $x>0$,
\[
W'(x) = 1+\frac{\lambda}{x}\left(2+\frac1x - \frac{\log(2\pi x)}{x}\right)                 
\ge 1+ \frac{\lambda}{x} (2-2\pi e^{-2}) \ge 1
\]
(indeed, the bracketed expression in the second term 
achieves minimum at the point $x=e^2/2\pi$ and
equals $2-2\pi{e^{-2}} > 1.1496633$ at this point).
It is strictly concave because, for all $x>0$,
\[
W''(x) = \frac{2\lambda}{x^3} \big(\log x-x + \log(\pi^2)-3/2\big) < 0
\]
(indeed $\log x-x \le -1$, for all $x>0$, whereas $\log(\pi^2)-3/2 
\approx 0.337877067 <1$); and so $W'(x)$ is a strictly decreasing function
with $W'(x)\to 1$ as $x \to \infty$.

\noindent{\bf The function $m_*$.} 
Define the strictly increasing function
\[
R(x) :=4 \lambda \log x + 2 \lambda+1, \quad x \ge 1,
\]
and notice that
\[
R(1) = 1+2\lambda > 1+\lambda \log(2\pi) = W(1).
\]
Hence, for each $x \ge 1$ there is a unique $m_*=m_*(x)$ such that
\[
W(m_*(x)) = R(x).
\]

\begin{lemma}
\label{applem}
With $m=\lfloor m_* - C_n/\log n \rfloor$ we have
\[
\frac{C_n}{\log n} < W(m_*) - W(m) 
< W'(1) \left(1+\frac{C_n}{\log n}\right)
\]
\end{lemma}
\begin{proof}
Write $W(m_*) - W(m) = W'(\eta) (m_*-m)$, for some $m < \eta < m_*$.
Since $1 < W'(x) < W'(1)$ for all $x > 1$
and $C_n/\log n \le m_*-m \le 1+C_n/\log n$ the inequalities follow.
\end{proof}

\begin{corollary}
\label{coroA2}
With $m=\lfloor m_* - C_n/\log n \rfloor$,
as $n \to \infty$,
\[
\log n - \frac{m}{4\lambda} = O(\log m).
\]
\end{corollary}
\begin{proof}
Use the previous lemma together with  $W(m) = m + O(\log m)$
and $W(m_*) = R(n) = 4 \lambda \log n + 2\lambda +1$.
\end{proof}

\begin{corollary}
\label{coroA3}
With $m=\lfloor m_* - C_n/\log n \rfloor$, as $n \to \infty$,
\[
\log n > \frac{1}{4 \lambda} m + \frac{1}{2} \log m - \frac{1}{2} -\frac{1}{4 \lambda}
+ o(1).
\]
\begin{proof}
substitute $W(m_*) = R(n) = 4 \lambda \log n + 2\lambda +1$
and $W(m) = m+ 2\lambda \log m + \frac{\lambda}{m} \log(2\pi m)$ in the first
inequality of Lemma \ref{applem} to get
\[
(4 \lambda \log n +2\lambda + 1) 
-\left(m+ 2\lambda \log m + \frac{\lambda}{m} \log(2\pi m)\right) 
> \frac{C_n}{\log n}.
\]
Since $\frac{\lambda}{m} \log(2\pi m)=o(1)$ and $\frac{C_n}{\log n}=o(1)$,
the inequality follows by rearranging terms.
\end{proof}
\end{corollary}

\begin{lemma}
\label{mstarapprox}
We have
\begin{equation}
\label{mstarclaim1}
\begin{split}
m_*(n) &= R(n) - 2\lambda \log R(n) + O\left(\frac{\log\log n}{\log n}\right)
\\
&=  4\lambda \log n+2\lambda+1
-2\lambda \log(4\lambda \log n+2\lambda+1) 
+ O\left(\frac{\log\log n}{\log n}\right),\, \text{ as $n \to \infty$.}
\end{split}
\end{equation}
\end{lemma}

\begin{proof}
Set 
\[
\tilde m(n) = R(n) - 2\lambda \log R(n).
\]
Claim \eqref{mstarclaim1} is equivalent to 
\begin{equation}
\label{mstarclaim2}
m_* = \tilde m + O\left(\frac{\log\log n}{\log n}\right).
\end{equation}
Since
\[
W(m_*) - W(\tilde m) = W'(\xi) (m^*-\tilde m)
\]
for some $\xi$ between $\tilde m$ and $m_*$,
and since $W'$ is bounded on the interval $[1,\infty)$,
claim \eqref{mstarclaim2}
is equivalent to
\begin{equation}
\label{mstarclaim3}
W(\tilde m)-W(m_*)  = O\left(\frac{\log\log n}{\log n}\right).
\end{equation}
This is easy because, by direct computation,
\begin{align*}
W(\tilde m)- W(m_*) &= W(\tilde m) - R
\\
&= \tilde m -R  + 2\lambda \log \tilde m  
+ \frac{\lambda}{\tilde m } \log(2\pi \tilde m )
\\
&= -2\lambda \log R + 2\lambda \log \tilde m 
+ \frac{\lambda}{\tilde m } \log(2\pi \tilde m )
%+  O\left(\frac{\log\log n}{\log n}\right)
\\
&= 2\lambda \log(\tilde m/R) + \frac{\lambda}{\tilde m } \log(2\pi \tilde m )
= O\left(\frac{\log\log n}{\log n}\right).
\end{align*}
The latter follows from the definitions of $\tilde m$ as a function of $R$
and the definition of $R$ as a function of $n$.
\end{proof}

\end{document}